\documentclass[11pt]{amsart}
\usepackage{amsopn}
\usepackage{amssymb, amscd}
\usepackage{multirow}
\usepackage{enumerate}

\usepackage{hyperref}

\newcommand{\nc}{\newcommand}

\nc{\fg}{\mathfrak{f} } \nc{\vg}{\mathfrak{v} } \nc{\wg}{\mathfrak{w} }
\nc{\zg}{\mathfrak{z} } \nc{\ngo}{\mathfrak{n} } \nc{\kg}{\mathfrak{k} }
\nc{\mg}{\mathfrak{m} } \nc{\bg}{\mathfrak{b} } \nc{\ggo}{\mathfrak{g} }
\nc{\ggob}{\overline{\mathfrak{g}} } \nc{\sog}{\mathfrak{so} }
\nc{\sug}{\mathfrak{su} } \nc{\spg}{\mathfrak{sp} } \nc{\slg}{\mathfrak{sl} }
\nc{\glg}{\mathfrak{gl} } \nc{\cg}{\mathfrak{c} } \nc{\rg}{\mathfrak{r} }
\nc{\hg}{\mathfrak{h} } \nc{\tg}{\mathfrak{t} } \nc{\ug}{\mathfrak{u} }
\nc{\dg}{\mathfrak{d} } \nc{\ag}{\mathfrak{a} } \nc{\pg}{\mathfrak{p} }
\nc{\sg}{\mathfrak{s} } \nc{\affg}{\mathfrak{aff} } \nc{\qg}{\mathfrak{q} }

\nc{\pca}{\mathcal{P}} \nc{\nca}{\mathcal{N}} \nc{\lca}{\mathcal{L}}
\nc{\oca}{\mathcal{O}} \nc{\mca}{\mathcal{M}} \nc{\tca}{\mathcal{T}}
\nc{\aca}{\mathcal{A}} \nc{\cca}{\mathcal{C}} \nc{\gca}{\mathcal{G}}
\nc{\sca}{\mathcal{S}} \nc{\hca}{\mathcal{H}} \nc{\bca}{\mathcal{B}}
\nc{\dca}{\mathcal{D}} \nc{\val}{\operatorname{val}}

\nc{\vp}{\varphi} \nc{\ddt}{\frac{d}{dt}} \nc{\dds}{\frac{d}{ds}}
\nc{\dpar}{\frac{\partial}{\partial t}} \nc{\im}{\mathrm{i}}

\nc{\SO}{\mathrm{SO}} \nc{\Spe}{\mathrm{Sp}} \nc{\Sl}{\mathrm{SL}}
\nc{\SU}{\mathrm{SU}} \nc{\Or}{\mathrm{O}} \nc{\U}{\mathrm{U}} \nc{\Gl}{\mathrm{GL}}
\nc{\Se}{\mathrm{S}} \nc{\Cl}{\mathrm{Cl}} \nc{\Spein}{\mathrm{Spin}}
\nc{\Pin}{\mathrm{Pin}} \nc{\G}{\mathrm{GL}_n(\RR)} \nc{\g}{\mathfrak{gl}_n(\RR)}

\nc{\RR}{{\Bbb R}} \nc{\HH}{{\Bbb H}} \nc{\CC}{{\Bbb C}} \nc{\ZZ}{{\Bbb Z}}
\nc{\FF}{{\Bbb F}} \nc{\NN}{{\Bbb N}} \nc{\QQ}{{\Bbb Q}} \nc{\PP}{{\Bbb P}} \nc{\OO}{{\Bbb O}}

\nc{\vs}{\vspace{.2cm}} \nc{\vsp}{\vspace{1cm}} \nc{\ip}{\langle\cdot,\cdot\rangle}
\nc{\ipp}{(\cdot,\cdot)} \nc{\la}{\langle} \nc{\ra}{\rangle} \nc{\unm}{\frac{1}{2}}
\nc{\unc}{\frac{1}{4}} \nc{\und}{\frac{1}{16}} \nc{\no}{\vs\noindent}
\nc{\lam}{\Lambda^2(\RR^n)^*\otimes\RR^n} \nc{\tangz}{{\rm T}^{\rm Zar}}
\nc{\nor}{{\sf n}}  \nc{\mum}{/\!\!/} \nc{\kir}{/\!\!/\!\!/}
\nc{\Ri}{\tfrac{4\Ric_{\mu}}{||\mu||^2}} \nc{\ds}{\displaystyle}
\nc{\ben}{\begin{enumerate}} \nc{\een}{\end{enumerate}} \nc{\f}{\frac}
\nc{\lb}{[\cdot,\cdot]} \nc{\isn}{\tfrac{1}{||v||^2}}
\nc{\gkp}{(\ggo=\kg\oplus\pg,\ip)} \nc{\ukh}{(\ug=\kg\oplus\hg,\ip)}
\nc{\tgkp}{(\tilde{\ggo}=\kg\oplus\pg,\ip)}
\nc{\wt}{\widetilde} \nc{\mm}{M}
\nc{\iop}{\mathtt{i}} \nc{\jop}{\mathtt{j}}

\nc{\Hess}{\operatorname{Hess}} \nc{\ad}{\operatorname{ad}}
\nc{\Ad}{\operatorname{Ad}} \nc{\rank}{\operatorname{rank}}
\nc{\Irr}{\operatorname{Irr}} \nc{\End}{\operatorname{End}}
\nc{\Aut}{\operatorname{Aut}} \nc{\Inn}{\operatorname{Inn}}
\nc{\Der}{\operatorname{Der}} \nc{\Ker}{\operatorname{Ker}}
\nc{\Iso}{\operatorname{Iso}} \nc{\Diff}{\operatorname{Diff}}
\nc{\Lie}{\operatorname{Lie}} \nc{\tr}{\operatorname{tr}} \nc{\dif}{\operatorname{d}}
\nc{\sen}{\operatorname{sen}} \nc{\modu}{\operatorname{mod}}
\nc{\CRic}{\operatorname{PP}} \nc{\Cric}{\operatorname{P}} \nc{\Ricci}{\operatorname{Ric}}
\nc{\sym}{\operatorname{sym}} \nc{\herm}{\operatorname{herm}} \nc{\symac}{\operatorname{sym^{ac}}}
\nc{\symc}{\operatorname{sym^{c}}} \nc{\scalar}{\operatorname{scal}}
\nc{\grad}{\operatorname{grad}} \nc{\ricci}{\operatorname{Rc}}
\nc{\Nor}{\operatorname{Norm}}  \nc{\ricc}{\operatorname{Rc^{c}}}
\nc{\Ricc}{\operatorname{Ric^{c}}} \nc{\ricac}{\operatorname{Rc^{ac}}}
\nc{\Ricac}{\operatorname{Ric^{ac}}} \nc{\Riem}{\operatorname{R}}
\nc{\riccig}{\operatorname{ric^{\gamma}}} \nc{\Rin}{\operatorname{M}}
\nc{\Le}{\operatorname{L}} \nc{\tang}{\operatorname{T}}
\nc{\level}{\operatorname{level}} \nc{\rad}{\operatorname{r}}
\nc{\abel}{\operatorname{ab}} \nc{\CH}{\operatorname{CH}}
\nc{\mcc}{\operatorname{mcc}} \nc{\Adj}{\operatorname{Adj}}
\nc{\Order}{\operatorname{O}}  \nc{\inj}{\operatorname{inj}} \nc{\proy}{\operatorname{proy}}
\nc{\vol}{\operatorname{vol}} \nc{\Diag}{\operatorname{Dg}}
\nc{\Spec}{\operatorname{Spec}} \nc{\Ima}{\operatorname{Im}} \nc{\Rea}{\operatorname{Re}}
\nc{\spann}{\operatorname{sp}}

\theoremstyle{plain}
\newtheorem{theorem}{Theorem}[section]
\newtheorem{proposition}[theorem]{Proposition}
\newtheorem{corollary}[theorem]{Corollary}
\newtheorem{lemma}[theorem]{Lemma}

\theoremstyle{definition}
\newtheorem{definition}[theorem]{Definition}

\theoremstyle{remark}
\newtheorem{remark}[theorem]{Remark}

\newtheorem{example}[theorem]{Example}

\title{Extremally Ricci pinched $G_2$-structures on Lie groups}

\author{Jorge Lauret} \author{Marina Nicolini}

\address{Universidad Nacional de C\'ordoba, FaMAF and CIEM, 5000 C\'ordoba, Argentina}
\email{lauret@famaf.unc.edu.ar} \email{mnicolini@famaf.unc.edu.ar}

\thanks{This research was partially supported by grants from CONICET, FONCYT and Universidad Nacional de C\'ordoba}

\begin{document}

\maketitle

\begin{abstract}
Only two examples of extremally Ricci pinched $G_2$-structures can be found in the literature and they are both homogeneous.  We study in this paper the existence and structure of such very special closed $G_2$-structures on Lie groups.  Strong structural conditions on the Lie algebra are proved to hold.  As an application, we obtain three new examples of extremally Ricci pinched $G_2$-structures and that they are all necessarily steady Laplacian solitons.  The deformation and rigidity of such structures are also studied.
\end{abstract}

\tableofcontents

\section{Introduction}\label{intro}

A $G_2$-{\it structure} on a $7$-dimensional differentiable manifold $M$ is a positive (or definite) differential $3$-form on $M$.  Each $G_2$-structure $\vp$ defines a Riemannian metric $g$ on $M$ together with an orientation and $(M,\vp)$ is called {\it homogeneous} if its automorphism group $\Aut(M,\vp):=\{ f\in\Diff(M):f^*\vp=\vp\}$ acts transitively on $M$.

As is well known, torsion-free (or parallel) $G_2$-structures (i.e. $d\vp=0$ and $d\ast\vp=0$) produce Ricci flat Riemannian metrics with holonomy contained in $G_2$.  Homogeneous torsion-free $G_2$-structures are therefore necessarily flat by \cite{AlkKml}.  In the case that $\vp$ is closed, the only torsion that survives is a $2$-form $\tau$ and one has that,
$$
d\vp=0, \qquad \tau=-\ast d\ast\vp, \qquad d\ast\vp=\tau\wedge\vp, \qquad d\tau=\Delta\vp.
$$
Closed $G_2$-structures play an important role as natural candidates to deform toward a torsion-free one via the {\it Laplacian flow} $\dpar\vp(t) = \Delta\vp(t)$, introduced back in 1992 by R. Bryant in \cite{Bry} (see \cite{Lty} for an account of recent advances).  In the homogeneous case, closed $G_2$-structures are only allowed on non-compact manifolds (see \cite{PdsRff}) and examples on non-solvable Lie groups were given in \cite{FinRff3}.

A closed $G_2$-structure is said to be {\it extremally Ricci-pinched} (ERP for short) when
$$
d\tau = \tfrac{1}{6}|\tau|^2\vp + \tfrac{1}{6}\ast(\tau\wedge\tau),
$$
one of the ways in which $d\tau$ can quadratically depend on $\tau$.  It is proved in \cite[(4.66)]{Bry} that this is the only way in the compact case.  In the homogeneous case, the only other possibility for a quadratic dependence is to have $d\tau = \frac{1}{7}|\tau|^2\vp$ (i.e.\ $\vp$ an eigenform of $\Delta$), though the existence of such structures is still an open problem (see \cite[Lemma 3.4]{LS-ERP} and \cite{Fields}).  ERP $G_2$-structures were introduced by R. Bryant in \cite[Remark 13]{Bry} and owe their name to the fact that they are precisely the structures at which equality holds in the following estimate for closed $G_2$-structures on a compact manifold $M$ obtained in \cite[Corollary 3]{Bry}\label{bryant}:
$$
\int_M \scalar^2 \ast 1 \leq 3\int_M |\Ricci|^2 \ast 1.
$$
This estimate does not hold in general in the homogeneous case, examples of closed $G_2$-structures on solvable Lie groups such that $\scalar^2 >  3 |\Ricci|^2$ were found in \cite{LS-ERP}.  In \cite{FinRff2}, it is proved that the Laplacian flow solution starting at an ERP $G_2$-structure $\vp$ is simply given by
$$
\vp(t) = \vp+c(t)d\tau, \qquad c(t)=\tfrac{6}{|\tau|^2}\left(e^{\frac{|\tau|^2}{6}t}-1\right),
$$
from which follows that the set of ERP $G_2$-structures is invariant under the Laplacian flow and the solutions are always eternal.

Until now, only two examples of ERP $G_2$-structures were known and they are both (locally) homogeneous: one on the homogeneous space $\Sl_2(\CC)\ltimes\CC^2/\SU(2)$ (see \cite[Example 1]{Bry}), or alternatively, on the solvable Lie group given in \cite[Section 6.3]{ClyIvn} (see also {\cite[Examples 4.13, 4.10]{LS-ERP}), and a second one on a unimodular solvable Lie group given in \cite[Example 4.7]{LS-ERP}.  It is worth highlighting that both examples are also steady Laplacian solitons, that is, they evolve under the Laplacian flow in the following silly way: there is a one-parameter family $f(t)\in\Diff(M)$ such that the Laplacian flow solution starting at $\vp$ is given by $\vp(t)=f(t)^*\vp$.

Motivated by this major lack of examples, we study in this paper left-invariant ERP $G_2$-structures on Lie groups, in which the $G_2$-structure can be identified with a positive $3$-form on the Lie algebra.  Our aim is to show that the condition produces quite strong structure constraints on the Lie algebra (see Section \ref{erp}).

We first introduce some notation.  Given a real vector space $\ggo$ with basis $\{ e_1,\dots,e_7\}$, we consider the positive $3$-form
\begin{align}
\vp=&e^{127}+e^{347}+e^{567}+e^{135}-e^{146}-e^{236}-e^{245} \notag \\
=&\omega_3\wedge e^3+\omega_4\wedge e^4+\omega_7\wedge e^7+e^{347}, \label{phi-intro}
\end{align}
where $\omega_7:=e^{12}+e^{56}$, $\omega_3:=e^{26}-e^{15}$ and $\omega_4:=e^{16}+e^{25}$, and let $\theta$ denote the usual representation of $\glg_4(\RR)$ on $\Lambda^2\RR^4$.  Two Lie groups endowed with $G_2$-structures $(G,\vp)$ and $(G',\vp')$ are called {\it equivariantly equivalent} if there is a Lie group isomorphism $f:G\rightarrow G'$ such that $\vp=f^*\vp'$.

We are now ready to state our main result (see Theorem \ref{erp-prop2} and Proposition \ref{erp-iff}).

\begin{theorem}\label{main}
Every Lie group endowed with a left-invariant ERP $G_2$-structure is equivariantly equivalent, up to scaling, to a $(G,\vp)$ with torsion
$\tau=e^{12}-e^{56}$, where $\vp$ is as in \eqref{phi-intro}, and the following conditions hold for the Lie algebra $\ggo$ of $G$:
\begin{itemize}
\item[(i)] $\hg:=\spann\{ e_1,\dots,e_6\}$ is a unimodular ideal.

\item[(ii)] $\ggo_0:=\spann\{ e_7,e_3,e_4\}$ is a Lie subalgebra  and $\ggo_1:=\spann\{ e_1,e_2,e_5,e_6\}$ is an abelian ideal of $\ggo$.  In particular, $\ggo=\ggo_0\ltimes\ggo_1$ and $\ggo$ is solvable.

\item[(iii)] $\hg_1:=\spann\{ e_3,e_4\}$ is an abelian subalgebra; in particular $\hg=\hg_1\ltimes\ggo_1$.

\item[(iv)] $\theta(\ad{e_7}|_{\ggo_1})\tau=\frac{1}{3}\omega_7$, $\theta(\ad{e_3}|_{\ggo_1})\tau=\frac{1}{3}\omega_3$ and $\theta(\ad{e_4}|_{\ggo_1})\tau=\frac{1}{3}\omega_4$.

\item[(v)] $\theta(\ad{e_7}|_{\ggo_1})\omega_7+\theta(\ad{e_3}|_{\ggo_1})\omega_3+\theta(\ad{e_4}|_{\ggo_1})\omega_4 = \tau +(\tr{\ad{e_7}|_{\ggo_0}})\omega_7$.
\end{itemize}
Conversely, if $\ggo$ satisfies (i)-(v), then $(G,\vp)$ is an ERP $G_2$-structure with torsion $\tau=e^{12}-e^{56}$.
\end{theorem}

As a first application, we obtain the following geometric consequence.

\begin{corollary}
Any left-invariant ERP $G_2$-structure on a Lie group is a steady Laplacian soliton and its underlying metric is an expanding Ricci soliton.
\end{corollary}

It is worth pointing out that the converse of the above corollary does not hold. Indeed, an example of a simply connected solvable Lie group endowed with a
steady Laplacian soliton that is not an ERP $G_2$-structure is exhibited in \cite{FinRff3}.

Structurally, it follows from Theorem \ref{main} that the Lie algebra $\ggo$ of any ERP $(G,\vp)$ is determined by the $2\times 2$ matrix
$A_1:=\ad{e_7}|_{\hg_1}$ and the three $4\times 4$ matrices $A:=\ad{e_7}|_{\ggo_1}$, $B=\ad{e_3}|_{\ggo_1}$, $C:=\ad{e_4}|_{\ggo_1}$. The Jacobi condition
is equivalent to
$$
[A,B]=aB+cC, \quad [A,C]=bB+dC, \quad [B,C]=0, \qquad
A_1=\left[\begin{smallmatrix} a&b\\ c&d\end{smallmatrix}\right].
$$
It must be stressed that conditions (iv) and (v) are really demanding on these matrices.

In Section \ref{cla-sec}, we exhibit three new examples of ERP $G_2$-structures on Lie groups and obtain further refinements for the algebraic structure of $\ggo$ by using the structural theorem on solvsolitons \cite[Theorem 4.8]{solvsolitons}.  We prove that there are only three possibilities for the nilradical $\ngo$ of $\ggo$ and that the following conditions must hold in each case:

\begin{itemize}
\item $\ngo=\ggo_1$: this is equivalent to $\ggo$ unimodular and one has that $A_1=0$, the matrices $A,B,C$ are all symmetric, they pairwise commute and $\left\{\sqrt{3}A, \sqrt{3}B, \sqrt{3}C\right\}$ is orthonormal.  In particular, $\ggo$ is isomorphic to the Lie algebra of \cite[Example 4.7]{LS-ERP}, a result previously obtained in \cite{FinRff2}.

\item $\ngo=\RR e_4\oplus\ggo_1$: $A,B$ are symmetric, $[A,B]=0$, $C$ is nilpotent and $a=b=c=0$.  We found two new examples in this case, with $\ngo$ $2$-step and $3$-step nilpotent, respectively.

\item $\ngo=\hg$: $A_1$ and $A$ are normal and $B,C$ nilpotent.  A new example is given with $\ngo$ $4$-step nilpotent.
\end{itemize}

Lastly, we study in Section \ref{deform-sec} deformations and rigidity of ERP $G_2$-structures on Lie groups by using the moving-bracket approach.  We have obtained that the five known examples are all rigid.

We believe that the present paper paves the way toward achieving a complete classification of ERP $G_2$-structures on Lie groups, which will be the object of further research.

\vs \noindent {\it Acknowledgements.} We are very grateful with Alberto Raffero for very helpful comments on a first version of the paper.

\section{Preliminaries}\label{preli}

\subsection{Linear algebra}
Given a real vector space $\ggo$ with basis $\{ e_1,\dots,e_7\}$, we consider the positive $3$-form
\begin{equation}\label{phiA}
\vp=\omega\wedge e^7+\rho^+=e^{127}+e^{347}+e^{567}+e^{135}-e^{146}-e^{236}-e^{245},
\end{equation}
where
$$
\omega:=e^{12}+e^{34}+e^{56}, \qquad \rho^+:=e^{135}-e^{146}-e^{236}-e^{245}.
$$
The usual notation $e^{ij\cdots}$ to indicate $e^i\wedge e^j\wedge\cdots$ will be freely used throughout the paper.  Note that $\omega\wedge\rho^+=0$.  We have that $\{ e_1,\dots,e_7\}$ is an oriented orthonormal basis with respect to the inner product $\ip$ and orientation $\vol$ determined by $\vp$, i.e.
\begin{equation}\label{phi-ip}
\la X,Y\ra\vol = \frac{1}{6} i_X(\vp)\wedge i_Y(\vp)\wedge\vp, \qquad\forall X,Y\in\ggo.
\end{equation}

The almost-complex structure $J$ defined on the subspace $\hg:=\spann\{ e_1,\dots,e_6\}$ by $\omega=\la J\cdot,\cdot\ra$ is given by $Je_i=e_{i+1}$, $i=1,3,5$, and we set
$$
\rho^-:= \ast_\hg\rho^+ = e^{145}+e^{136}+e^{235}-e^{246}.
$$
Let $\theta:\glg(\hg)\longrightarrow\End(\Lambda^k\hg^*)$ denote the representation obtained as the derivative of the natural left $\Gl(\hg)$-action on
each $\Lambda^k\hg^*$ (i.e.\ $h\cdot\alpha=\alpha(h^{-1}\cdot,\dots,h^{-1}\cdot)$), which is given for each $B\in\glg(\hg)$ by,
$$
\theta(B)\gamma = \ddt\Big|_0 e^{tB}\cdot\gamma = -\left(\gamma(B\cdot,\dots,\cdot) + \dots +\gamma(\cdot,\dots,B\cdot)\right), \qquad\forall\gamma\in\Lambda^k\hg^*.
$$
The following technical lemma contains some useful information on the linear algebra involved in subsequent computations.

\begin{lemma}\label{tech1}
Let $\ast:\Lambda^k\ggo^*\longrightarrow\Lambda^{7-k}\ggo^*$ and $\ast_\hg:\Lambda^k\hg^*\longrightarrow\Lambda^{6-k}\hg^*$ be the Hodge star operators determined by the ordered bases $\{ e_1,\dots,e_7\}$ and $\{ e_1,\dots,e_6\}$, respectively.

\begin{itemize}
\item[(i)] $\ast\gamma=\ast_\hg\gamma\wedge e^7$, for any $\gamma\in\Lambda^k\hg^*$.

\item[(ii)] $\ast(\gamma\wedge e^7)=(-1)^k\ast_\hg\gamma$, for any $\gamma\in\Lambda^k\hg^*$.

\item[(iii)] $\ast_\hg\omega=\unm\omega\wedge\omega$ and $\ast_\hg(\omega\wedge\omega)=2\omega$.

\item[(iv)] $\ast^2=id$ and $\ast_\hg^2=(-1)^k id$ on $\Lambda^k\hg^*$.

\item[(v)] $\ast\vp = \unm\omega\wedge\omega + \rho^-\wedge e^7 = e^{3456}+e^{1256}+e^{1234}-e^{2467}+e^{2357}+e^{1457}+e^{1367}$.
\item[(vi)] $\theta(A)\ast_\hg + \ast_\hg\theta(A^t) = -(\tr{A})\ast_\hg$ on $\Lambda\hg^*$, for any $A\in\glg(\hg)$.
\end{itemize}
\end{lemma}

\begin{proof}
Parts (i)-(v) follow easily (see e.g.\ \cite[Lemmas 5.11, 5.12]{LF}) and to prove part (vi), we first recall that
$$
\alpha\wedge\ast_\hg\beta = \la\alpha,\beta\ra \ast_\hg 1, \qquad \ast_\hg 1=e^1\wedge\dots\wedge e^6, \qquad\forall\alpha,\beta\in\Lambda^k\hg^*.
$$
Thus, for any $\alpha\in\Lambda^p\hg^*$ and $\beta\in\Lambda^{6-p}\hg^*$, one has
\begin{align*}
\la\alpha,\theta(A)\ast_\hg\beta\ra\ast_\hg 1 =& \la\theta(A^t)\alpha,\ast_\hg\beta\ra\ast_\hg 1  = \theta(A^t)\alpha\wedge\ast_\hg^2\beta = (-1)^p\theta(A^t)\alpha\wedge\beta \\
=& (-1)^{p+1}(\tr{A})\alpha\wedge\beta + (-1)^{p+1}\alpha\wedge\theta(A^t)\beta  \\
=& -(\tr{A})\alpha\wedge\ast_\hg\ast_\hg\beta-\alpha\wedge\ast_\hg\ast_\hg\theta(A^t)\beta \\
=& \la\alpha,-(\tr{A})\ast_\hg\beta-\ast_\hg\theta(A^t)\beta\ra\ast_\hg 1,
\end{align*}
concluding the proof of the lemma.
\end{proof}

Recall that $\theta(B)$ is a derivation of the algebra $\Lambda\hg^*$ and that $\theta(B)e^{1\cdots 6}=-(\tr{B})e^{1\cdots 6}$.  We consider the $14$-dimensional simple Lie group
$$
G_2:=\{ h\in\Gl_7(\RR):h\cdot\vp=\vp\}\subset\SO(7),
$$
where $\vp$ is as in \eqref{phiA}.  Let $\ggo_2$ denote the Lie algebra of $G_2$.  The spaces of $2$-forms and $3$-forms on $\ggo$ respectively decompose into irreducible $G_2$-representations as follows,
$$
\Lambda^2\ggo^* = \Lambda^2_{7}\ggo^*\oplus\Lambda^2_{14}\ggo^*, \qquad \Lambda^3\ggo^* = \Lambda^3_{1}\ggo^*\oplus\Lambda^3_{7}\ggo^*\oplus\Lambda^3_{27}\ggo^*,
$$
where subscript numbers are the dimensions.  A description of each of these irreducible components (see e.g.\ \cite[(2.14)]{Bry}) can be obtained by considering different suitable $G_2$-equivariant linear maps.   For example, the kernel of the map $\Lambda^2\ggo^*\longrightarrow\Lambda^6\ggo^*$, $\alpha\mapsto\alpha\wedge\ast\vp$ must be $\Lambda^2_{14}\ggo^*$.  On the other hand, the map $\Lambda^2\ggo^*\longrightarrow\Lambda^2\ggo^*$, $\alpha\mapsto\ast(\alpha\wedge\vp)$ is necessarily a multiple of the identity and one obtains that such a multiple is $-1$ by evaluating at $e^{12}-e^{34}$.  This implies that
\begin{equation}\label{L214}
\Lambda^2_{14}\ggo^*=\{\alpha\in\Lambda^2\ggo^*:\alpha\wedge\ast\vp=0\} = \{\alpha\in\Lambda^2\ggo^*:\alpha\wedge\vp=-\ast\alpha\}.
\end{equation}
Since $\Lambda^2_{14}\ggo^*$ is, as a $G_2$-representation, equivalent to the adjoint representation $\ggo_2$, any nonzero $\tau\in\Lambda^2_{14}\ggo^*$
can be diagonalized, in the sense that there exists an oriented orthonormal basis $\{ e_1,\dots,e_7\}$ of $\ggo$ such that $\vp$ is as in \eqref{phiA} and
\begin{equation}\label{tau-normal}
\tau=a\, e^{12}+b\, e^{34}+c\, e^{56}, \qquad a+b+c=0, \qquad a\geq b\geq 0> c.
\end{equation}
In particular,
\begin{align}
\tau\wedge\tau&=2ab\, e^{1234}+2ac\, e^{1256}+2bc\, e^{3456}, \notag\\
\tau\wedge\tau\wedge\tau&=6abc\, e^{123456}, \label{prop-tn}\\
|\tau\wedge\tau|&=|\tau|^2=a^2+b^2+c^2.  \notag
\end{align}

\subsection{The Lie group $G_\mu$}\label{Gmu}
Let $\ggo$ be a Lie algebra of dimension $7$ and assume that $\ggo$ has a $6$-dimensional ideal $\hg$.  Consider a basis $\{ e_1,\dots,e_7\}$ of $\ggo$ such that $\hg=\spann\{e_1,\dots,e_6\}$, so $\ggo=\hg\rtimes\RR e_7$.  The Lie bracket $\mu$ of $\ggo$ is therefore given by
\begin{equation}\label{mu}
\mu=\lambda+\mu_A,
\end{equation}
where $\lambda$ is the Lie bracket of $\hg$ (extended to $\ggo$ by $\lambda(\ggo,e_7)=0$) and $\mu_A$ is the Lie bracket defined for some $A\in\Der(\hg)$ by
$$
\mu_A(e_7,v)=Av, \qquad \mu_A(v,w)=0, \qquad \forall v,w\in\hg.
$$
Let $G_\mu$ denote the simply connected Lie group with Lie algebra $(\ggo,\mu)$.  Note that $G_\mu$ is solvable if and only if the Lie algebra $(\hg,\lambda)$ is solvable, and it is nilpotent if and only if $(\hg,\lambda)$ is nilpotent and $A$ is a nilpotent linear map.  Denote by $H_\lambda$ the simply connected Lie group with Lie algebra $(\hg,\lambda)$ and by $G_A$ the simply connected Lie group with Lie algebra $(\ggo,\mu_A)$.

Some properties of the differentials of forms on these Lie groups are given in the following lemma.

\begin{lemma}\label{tech3}
Let $d_\mu,d_\lambda,d_A$ denote the differentials of left-invariant $k$-forms on the Lie groups $G_\mu$, $H_\lambda$ and $G_A$, respectively.
\begin{itemize}
\item[(i)] $d_\mu=d_\lambda+d_A$, for any $\gamma=\alpha+\beta\wedge e^7\in\Lambda^k\ggo^*$, $\alpha\in\Lambda^k\hg^*$, $\beta\in\Lambda^{k-1}\hg^*$,
$$
d_\mu\gamma=d_\lambda\alpha + d_\lambda\beta\wedge e^7 + d_A\alpha,
$$
and $d_A\alpha=(-1)^k\theta(A)\alpha\wedge e^7$.

\item[(ii)] $d_\mu e^7=0$, $d_\lambda e^7=0$, $d_A e^7=0$ and $d_A(\alpha\wedge e^7)=0$, for all $\alpha\in\Lambda^k\hg^*$.

\item[(iii)] $d_\lambda\circ\theta(D)=\theta(D)\circ d_\lambda$ for any $D\in\Der(\hg)$.

\item[(iv)] $d_\mu\ast e^i = (-1)^i\tr(\ad_\mu{e_i}) e^{1\dots7}$, for any $i=1,\dots,7$.

\item[(v)] $d_\mu^*|_{\Lambda^k\ggo^*}=(-1)^{k+1}\ast d_\mu\ast$ and $d_\lambda^*|_{\Lambda^k\hg^*}=\ast d_\lambda\ast$ for any $k$.
\end{itemize}
\end{lemma}

\begin{proof}
Parts (i) and (ii) clearly hold (see e.g.\ \cite[Lemma 5.12]{LF}), parts (iv) and (v) are straightforward computations and part (iii) follows from the fact that $\theta(D)$ is precisely minus the Lie derivative $\lca_{X_D}$, where $X_D$ is the vector field on $H_\lambda$ attached to $D$.
\end{proof}

\subsection{Subgroups of $G_2$} \label{subg-sec}
In our study of ERP $G_2$-structures in Section \ref{erp}, we need to compute the stabilizer of the $2$-form $\tau:=e^{12}-e^{56}$ in $G_2$, as well as inside the subgroup $U_\hg\subset G_2$ leaving $\hg$ invariant.  It is well known (see e.g.\ \cite[Lemma 2.2.2]{VanMnr}) that
\begin{equation}\label{Uh}
U_\hg:=\left\{ h\in G_2 : h(\hg)\subset\hg\right\} = \left[\begin{array}{c|c} 1&\\\hline &\SU(3)\end{array}\right] \cup \left[\begin{array}{c|c} 1&\\\hline &\SU(3)\end{array}\right]\tilde{g},
\end{equation}
where $\tilde{g} = \Diag(-1,1,-1,1,-1,1,-1)$ and $\SU(3)$ is defined by using $J$.  Any matrix in this section will be written in terms of the basis $\{ e_7,e_3,e_4,e_1,e_2,e_5,e_6\}$.

\begin{lemma}\label{subg}
The subgroup of $G_2$
$$
U_{\hg,\tau}:=\left\{ h\in G_2 : h(\hg)\subset\hg, \; h\cdot\tau=\tau\right\},
$$
is given by $U_{\hg,\tau}=U_0\cup U_0g$, where
$$
U_0:=\left\{\left[\begin{smallmatrix}
1&&& \\ &h_1&& \\ &&h_2&\\ &&&h_3
\end{smallmatrix}\right] : h_i\in\SO(2), \; h_1h_2h_3=I\right\}, \quad
g:=\left[\begin{smallmatrix}
-1&&&&&& \\ &1&&&&& \\ &&-1&&&&\\ &&&&&1&0\\ &&&&&0&-1\\ &&&-1&0&&\\ &&&0&1&&
\end{smallmatrix}\right].
$$
\end{lemma}

\begin{proof}
Since for any $h\in\Or(7)$, $h\cdot\tau=\tau$ if and only if
$$
hJ_1h^{-1}=J_1, \qquad J_1:=\left[\begin{smallmatrix}
0&&&&&& \\ &0&&&&& \\ &&0&&&&\\  &&&0&-1&&\\ &&&1&0&&\\ &&&&&0&1\\ &&&&&-1&0
\end{smallmatrix}\right],
$$
it is not hard to see that
$$
U_0 = U_{\hg,\tau} \cap \left[\begin{array}{c|c} 1&\\\hline &\SU(3)\end{array}\right].
$$
On the other hand, by using that $\tilde{g}\in G_2$ and $\tilde{g}\cdot\tau=-\tau$ (see \eqref{Uh}), we obtain that
$$
g = \left[\begin{smallmatrix}
1&&&&&& \\ &1&&&&& \\ &&1&&&&\\  &&&&&1&0\\ &&&&&0&1\\ &&&-1&0&&\\ &&&0&-1&&
\end{smallmatrix}\right]\tilde{g} \in U_{\hg,\tau}\cap\left[\begin{array}{c|c} 1&\\\hline &\SU(3)\end{array}\right]\tilde{g}.
$$
Now if $h=f\tilde{g}\in U_{\hg,\tau}$, where $fe_7=e_7$ and $f_0:=f|_\hg\in\SU(3)$, then
$$
f_0=\left[\begin{smallmatrix} f_1&&\\ &0&f_2\\ &f_3&0 \end{smallmatrix}\right], \qquad f_1f_2f_3=-1,
$$
as $f_0$ must commute with $J|_\hg$ and $J_1|_\hg$, and therefore,
$$
h = \left[\begin{smallmatrix} 1&&&\\ &f_1&&\\ &&f_2&\\ &&&-f_3 \end{smallmatrix}\right]g \in U_0g, \quad\mbox{that is}, \quad
U_{\hg,\tau}\cap\left[\begin{array}{c|c} 1&\\\hline &\SU(3)\end{array}\right]\tilde{g} = U_0g,
$$
concluding the proof.
\end{proof}

Other subgroups of $G_2$ we will need to consider are
\begin{equation}\label{Ug2}
U_{\ggo_1}:=\left\{ h\in G_2 : h(\ggo_1)\subset\ggo_1\right\},
\end{equation}
where $\ggo_1:=\spann\{ e_1,e_2,e_5,e_6\}$, whose Lie algebra is well-known (see e.g.\ \cite{VanMnr}) to be given by
$$
\ug_{\ggo_1}= \left\{\left[\begin{smallmatrix}
0&c&-b&&&& \\ -c&0&a&&&& \\ b&-a&0&&& \\
&&&0&-d&b-e&-f\\ &&&d&0&-c+f&-e\\ &&&-b+e&c-f&0&-a+d\\ &&&f&e&a-d&0
\end{smallmatrix}\right] : a,b,c,d,e,f\in\RR\right\},
$$
and the corresponding subgroup stabilizing $\tau$,
\begin{equation}\label{Ug}
U_{\ggo_1,\tau}:=\left\{ h\in G_2 : h(\ggo_1)\subset\ggo_1, \; h\cdot\tau=\tau\right\},
\end{equation}
with Lie algebra,
$$
\ug_{\ggo_1,\tau}= \left\{\left[\begin{smallmatrix}
0&c&-b&&&& \\ -c&0&a&&&& \\ b&-a&0&&& \\
&&&0&-d&\unm b&-\unm c\\ &&&d&0&-\unm c&-\unm b\\ &&&-\unm b&\unm c&0&-a+d\\ &&&\unm c&\unm b&a-d&0
\end{smallmatrix}\right] : a,b,c,d\in\RR\right\}.
$$

\section{Closed $G_2$-structures}\label{closed}

A $G_2$-{\it structure} on a $7$-dimensional differentiable manifold $M$ is a differential $3$-form $\vp\in\Omega^3M$ such that $\vp_p$ is {\it positive} on $T_pM$ for any $p\in M$, that is, $\vp_p$ can be written as in \eqref{phiA} with respect to some basis $\{ e_1,\dots,e_7\}$ of $T_pM$.  Each $G_2$-structure $\vp$ defines a Riemannian metric $g$ on $M$ and an orientation $\vol\in\Omega^7M$ (unique up to scaling) as in \eqref{phi-ip}.  Thus $\vp$ also determines a Hodge star operator $\ast:\Omega M\longrightarrow\Omega M$ and the {\it Hodge Laplacian operator} $\Delta:\Omega^kM\longrightarrow\Omega^kM$, $\Delta:=d^*d+dd^*$, where $d^*:\Omega^{k+1}M\longrightarrow\Omega^kM$, $d^*=(-1)^{k+1}\ast d\ast$, is the adjoint of $d$.  The {\it torsion forms} of a $G_2$-structure $\vp$ on $M$ are the components of the {\it intrinsic torsion} $\nabla\vp$, where $\nabla$ is the Levi-Civita connection of the metric $g$.  They can be defined as the unique differential forms $\tau_i\in\Omega^iM$, $i=0,1,2,3$, such that
\begin{equation}\label{dphi}
d\vp=\tau_0\ast\vp+3\tau_1\wedge\vp+\ast\tau_3, \qquad d\ast\vp=4\tau_1\wedge\ast\vp+\tau_2\wedge\vp.
\end{equation}

Two manifolds endowed with $G_2$-structures $(M,\vp)$ and $(M',\vp')$ are called {\it equivalent} if there exists a diffeomorphism $f:M\longrightarrow M'$ such that $\vp=f^*\vp'$.

In the case of a closed $G_2$-structure $\vp$ on a $7$-manifold $M$, the only torsion that survives is a the $2$-form $\tau:=\tau_2$ and one therefore has that,
\begin{equation}\label{tauphi}
d\vp=0, \qquad \tau=d^*\vp=-\ast d\ast\vp, \qquad d\ast\vp=\tau\wedge\vp, \qquad d\tau=\Delta\vp.
\end{equation}
In particular, $\vp$ is torsion-free (or parallel) if and only if $\tau=0$.   Since $\tau\in\Omega^2_{14}M$ (see e.g.\  \cite[Proposition 1]{Bry}), all the useful conditions \eqref{L214}-\eqref{prop-tn} hold for $\tau$ at each $p\in M$.

In this paper, we study left-invariant $G_2$-structures on Lie groups, which allows us to work at the Lie algebra level as in Section \ref{preli}.  The $G_2$-structure is determined by a positive $3$-form on the Lie algebra $\ggo$, which will be most of the times the one given in \eqref{phiA}.  Two Lie groups endowed with left-invariant $G_2$-structures $(G,\vp)$ and $(G',\vp')$ are called {\it equivariantly equivalent} if there exists a Lie group isomorphism $F:G\longrightarrow G'$ such that $\vp=f^*\vp'$, where $f:=dF|_e:\ggo\longrightarrow\ggo'$ is the corresponding Lie algebra isomorphism.

\begin{definition}\label{def-GmuG2}
$(G_\mu,\vp)$ is the Lie group $G_\mu$ defined in Section \ref{Gmu} endowed with the left-invariant $G_2$-structure determined by the positive $3$-form $\vp$ on $\ggo$ given in \eqref{phiA}.
\end{definition}
Recall from \eqref{mu} that $G_\mu$ depends only on the Lie bracket $\lambda$ on $\hg=\spann\{ e_1,\dots,e_6\}$ and the map $A\in\Der(\hg,\lambda)$.

\begin{proposition}\label{closed-mu}
Any Lie group endowed with a left-invariant closed $G_2$-structure is equivariantly equivalent to $(G_\mu,\vp)$ for some $\mu=\lambda+\mu_A$ such that the ideal $(\hg,\lambda)$ is unimodular.
\end{proposition}

\begin{remark}
If the Lie group is not unimodular, then $\hg=\{ X\in\ggo:\tr{\ad_\mu{X}}=0\}$.  On the other hand, the pair $(\omega,\rho^+)$ defines an $\SU(3)$-structure on the Lie algebra $\hg$.
\end{remark}

\begin{proof}
Let $(G,\psi)$ be a Lie group $G$ endowed with a closed $G_2$-structure $\psi$.  If $\ggo$ is not unimodular, then we can take the codimension-one ideal $\hg$ of $\ggo$ as above and in the case when
$\ggo$ is unimodular, it follows from the classification obtained in \cite[Main Theorem]{FinRff3} that there exists a codimension-one ideal $\hg$.  Thus there exists an orthonormal basis $\{e_1,\dots,e_7\}$ of $\ggo$ such that $\hg=\spann\{e_1,\dots,e_6\}$ and $\vp$ can be written as in \eqref{phiA}.  It therefore follows that if $\lambda:=\mu|_{\hg\times\hg}$ and $A:=\ad e_7|_\hg$, then the Lie bracket $\mu$ of $\ggo$ is given by $\mu=\lambda+\mu_A$, concluding the proof.
\end{proof}

We therefore focus from now on on $G_2$-structures of the form $(G_\mu,\vp)$.  According to Lemma \ref{tech3}, (i), for $(G_\mu,\vp)$ one has that,
\begin{align}
d_\mu\vp =& d_\lambda\rho^+ + d_\lambda\omega\wedge e^7 - \theta(A)\rho^+\wedge e^7, \label{dphi}  \\
d_\mu\ast\vp =& d_\lambda\omega\wedge\omega + d_\lambda\rho^-\wedge e^7 + \theta(A)\omega\wedge\omega\wedge e^7. \label{destphi}
\end{align}
Thus $(G_\mu,\vp)$ is closed if and only if the following two conditions hold:
\begin{equation}\label{muclosed}
d_\lambda\omega=\theta(A)\rho^+, \qquad d_\lambda\rho^+=0.
\end{equation}
We now compute the torsion in terms of $\lambda$ and $A$, which is the only data varying.

\begin{proposition}\label{tauLA}
The torsion $2$-form $\tau_\mu$ of a closed $G_2$-structure $(G_\mu,\vp)$ is given by $\tau_\mu=\tau_\lambda+\tau_A$, where
$$
\tau_\lambda:=-\ast_\hg(d_\lambda\omega\wedge\omega)\wedge e^7 -\ast_\hg d_\lambda\rho^-, \qquad \tau_A:=(\tr{A})\omega + \theta(A^t)\omega.
$$
Furthermore, $d_\lambda\omega\wedge\omega = -\theta(A)\omega\wedge\rho^+$.
\end{proposition}

\begin{proof}
It follows from \eqref{phiA} that
\begin{align*}
\ast d_A\ast\vp =& \ast\left(\unm\theta(A)(\omega\wedge\omega)\wedge e^7\right) = \ast\left(\theta(A)\ast_\hg\omega\wedge e^7\right) \\
=& \ast\left(-(\tr{A})\ast_\hg\omega\wedge e^7-\ast_\hg\theta(A^t)\omega\wedge e^7\right) = -(\tr{A})\omega-\theta(A^t)\omega,
\end{align*}
and
$$
d_\lambda\omega\wedge\omega = \theta(A)\rho^+\wedge\omega = \theta(A)\omega\wedge\rho^+.
$$
Using Lemma \ref{tech1}, (v) we now compute,
\begin{align*}
d_\mu\ast\vp =& d_\lambda\omega\wedge\omega + d_\lambda\rho^-\wedge e^7 + d_A\ast\vp, \\
\ast d_\mu\ast\vp =& \ast(d_\lambda\omega\wedge\omega) + \ast_\hg d_\lambda\rho^- + \ast d_A\ast\vp \\
=& \ast_\hg(d_\lambda\omega\wedge\omega)\wedge e^7 + \ast_\hg d_\lambda\rho^- - (\tr{A})\omega - \theta(A^t)\omega,
\end{align*}
from which the desired formula follows.
\end{proof}

Straightforwardly, one obtains that the above proposition and Lemma \ref{tech3} give that for any closed $G_2$-structure $(G_\mu,\vp)$,
\begin{align}
d_\mu\tau_\mu =& - d_\lambda \ast_\hg (d_\lambda \omega \wedge \omega) \wedge e^7 - d_\lambda \ast_\hg d_\lambda \ast_\hg \rho^+ - \theta(A) \ast_\hg d_\lambda \rho^- \wedge e^7 \label{DeltaLA} \\
&+ (\tr{A}) d_\lambda \omega+ (\tr{A}) \theta(A) \omega \wedge e^7 + \theta(A) \theta(A^t) \omega \wedge e^7 + d_\lambda\theta(A^t)\omega. \notag
\end{align}

The following result shows that two left-invariant $G_2$-structures on non-isomorphic Lie groups can indeed be equivalent, in spite of they are not equivariantly equivalent.  This generalizes \cite[Proposition 5.6]{LF} beyond the almost-abelian case and the proof also follows the lines of  \cite[Proposition 2.5]{Hbr2}.

\begin{proposition}\label{equivA}
Let $(G_\mu,\vp)$ be a $G_2$-structure as above with $\mu=\lambda+\mu_A$.   If $D\in\sug(3)\cap\Der(\hg,\lambda)$, $[D,A]=0$ and we set $\mu_1:=\lambda+\mu_{A+D}$, then the $G_2$-structures $(G_\mu,\vp)$ and $(G_{\mu_1},\vp)$ are equivalent.
\end{proposition}

\begin{remark}\label{equivA-rem}
The hypothesis on the matrix $D$ means precisely that
$$
\left[\begin{array}{cc}
 D & 0 \\
0 & 0
\end{array}\right] \in \ggo_2\cap\Der(\ggo,\mu), \; \ggo_2\cap\Der(\ggo,\mu_1).
$$
Note that the Lie groups $G_\mu$, $G_{\mu_1}$ are in general not isomorphic.  For instance, if $\mu$ is not unimodular, then the spectra of $D$ and $A$ must coincide up to scaling in order to $\mu$ and $\mu_1$ be isomorphic.
\end{remark}

\begin{proof}
Let $\ggo_\mu$ denote the Lie algebra $(\ggo,\mu)$ of $G_\mu$.  We  consider the Lie group
$$
F:=\Aut(G_{\mu_1})\cap\Aut(G_{\mu_1},\vp)\simeq\Aut(\ggo_{\mu_1})\cap G_2,
$$
with Lie algebra $\fg:=\Der(\ggo_{\mu_1})\cap\ggo_2$, the homomorphism $\alpha:\ggo_{\mu_1}\longrightarrow\fg$ defined by
$$
\alpha(e_7)=
\left[\begin{array}{cc}
 -D  & 0 \\
0 & 0
\end{array}\right], \qquad \alpha|_\hg\equiv 0,
$$
and denote also by $\alpha$ the corresponding Lie group homomorphism $G_{\mu_1}\longrightarrow F$.   If $L:G_{\mu_1}\longrightarrow\Aut(G_{\mu_1},\vp)$ is the left-multiplication morphism, then
$$
G_1:=\{ L_s\circ\alpha(s):s\in G_{\mu_1}\}\subset\Aut(G_{\mu_1},\vp),
$$
is a subgroup.  Indeed, using that $G_{\mu_1}=\exp{\RR e_7}\ltimes\exp{\hg}$, one has for $s=ah$, $t=bg$ that,
\begin{align*}
L_s\circ\alpha(s)\circ L_t\circ\alpha(t) =& L_s\circ L_{\alpha(s)(t)}\alpha(s)\alpha(b) = L_{s\alpha(s)(t)}\alpha(s)\alpha(b\alpha(a)(g)) \\
=& L_{s\alpha(s)(t)}\alpha(s)\alpha(\alpha(a)(b)\alpha(a)(g)) = L_{s\alpha(s)(t)}\alpha(s\alpha(s)(t)).
\end{align*}
Thus $G_1$ is a connected and closed Lie subgroup of $\Aut(G_{\mu_1},\vp)$ since $s\mapsto L_s\circ\alpha(s)$ is continuous and proper.  Furthermore, $G_1$ acts simply and transitively on $G_{\mu_1}$ by automorphisms of $\vp$, so the diffeomorphism $f:G_1\longrightarrow G_{\mu_1}$, $f(L_s\circ\alpha(s)) := (L_s\circ\alpha(s))(e)=s$ defines an equivalence between the left-invariant $G_2$-structures $(G_1,f^*\vp)$ and $(G_{\mu_1},\vp)$.  On the other hand, the Lie algebra of $G_1$ is given by
$$
\ggo_0:=\left\{ dL|_eX+\alpha(X):X\in\ggo\right\} \subset \Lie\left(\Aut(G_{\mu_1},\vp)\right),
$$
and if $X=X_\hg+ae_7$, $Y=Y_\hg+be_7$ belong to $\ggo$, then
\begin{align*}
[dL|_eX+&\alpha(X),dL|_eY+\alpha(Y)] = dL|_e\mu_1(X,Y) + dL|_e\alpha(X)Y - dL|_e\alpha(Y)X + \alpha([X,Y]) \\
=& dL|_e\left(a(A+D)Y_\hg-b(A+D)X_\hg +\lambda(X_\hg,Y_\hg)- aDY_\hg + bDX_\hg + 0\right) \\
=& dL|_e\left(aDY_\hg-bDX_\hg +\lambda(X_\hg,Y_\hg)\right)
= dL|_e\mu(X,Y) = (dL|_e+\alpha)\mu(X,Y).
\end{align*}
This shows that $df|_e^{-1}=dL|_e+\alpha:\ggo_\mu\longrightarrow\ggo_0$ is a Lie algebra isomorphism and so $(G_\mu,\vp)$ is equivalent to $(G_1,f^*\vp)$, concluding the proof.
\end{proof}

\begin{remark}
By replacing $\vp$ by an inner product $\ip$ on $\ggo$, $\Aut(G_{\mu_1},\vp)$ by $\Iso(G_{\mu_1},\ip)$ and $G_2$ by $\Or(\ggo,\ip)$, the following Riemannian version can be proved in exactly the same way as above for any dimension: $(G_\mu,\ip)$ is isometric to $(G_{\mu_1},\ip)$ for any $\mu=\lambda+\mu_A$, $\mu_1=\lambda+\mu_{A+D}$ such that $D\in\sog(\hg,\ip)\cap\Der(\hg,\lambda)$ and $[D,A]=0$.
\end{remark}

As an application of Proposition \ref{equivA}, one obtains that the one-parameter family of extremally Ricci pinched $G_2$-structures given in \cite[Example 6.4]{FinRff2} is pairwise equivalent.

\subsection{ERP $G_2$-structures}\label{erp-subsec}
The following nice estimate for a closed $G_2$-structure $\vp$ on a compact manifold $M$ was proved by R. Bryant (see \cite[Corollary 3]{Bry}):
$$
\int_M \scalar^2 \ast 1 \leq 3\int_M |\Ricci|^2 \ast 1,
$$
and equality holds if and only if
\begin{equation}\label{erp-def-2}
d\tau = \frac{1}{6}|\tau|^2\vp + \frac{1}{6}\ast(\tau\wedge\tau).
\end{equation}
The factor of $3$ on the right hand side, being much smaller than $7$, shows that the metric is always far from being Einstein.

\begin{definition}\label{erp-def}
The distinguished closed $G_2$-structures for which condition \eqref{erp-def-2} holds and $\tau\ne 0$ were called {\it extremally Ricci-pinched} (ERP for short) in \cite[Remark 13]{Bry}.
\end{definition}

We begin with some general results on such structures.

\begin{proposition}\label{erp-prop}\cite{Bry}
Let $(M,\vp)$ be a manifold endowed with an ERP $G_2$-structure and assume that it is locally homogeneous.  Then,
\begin{itemize}
\item[(i)] $\tau\wedge\tau\wedge\tau=0$.
\item[(ii)] $d(\tau\wedge\tau)=0$.
\item[(iii)] $d\ast(\tau\wedge\tau)=0$.
\item[(iv)] $\Ricci|_P=-\frac{1}{6}|\tau|^2I$, $\Ricci|_Q=0$ and $\la\Ricci P,Q\ra=0$, where
$$
P:=\{ X\in TM:\iota_X(\tau\wedge\tau)=0\}, \qquad Q:=\{ X\in TM:\iota_X\ast(\tau\wedge\tau)=0\},
$$
and $\dim{P}=3$, $\dim{Q}=4$.
\end{itemize}
\end{proposition}

\begin{proof}
Parts (i), (ii) and (iii) follow from \cite[(4.53)]{Bry}, \cite[(4.55)]{Bry} and \cite[(4.51)]{Bry}, respectively, and the fact that $d|\tau|^2=0$ since
$M$ is locally homogeneous.  If we write $\tau$ as in \eqref{tau-normal} at each $p\in M$, then $b=0$ and $c=-a$ must hold, since
$\tau\wedge\tau\wedge\tau=0$ by (i), and so $\tau=a (e^{12}- e^{56})$. To prove part (iv), we consider the formula  given in \cite[(16)]{LS-ERP} for
$q=\tfrac{1}{6}$, so in terms  of the ordered basis $\{ e_7,e_3,e_4,e_1,e_2,e_5,e_6\}$,
\begin{align*}
\Ricci=&-\frac{1}{6}|\tau|^2I-\frac{1}{3}\tau^2=-\frac{1}{3}a^2I-\frac{1}{3}\Diag(-a^2,-a^2,0,0,-a^2,-a^2,0)\\
      =&-\frac{a^2}{3}\Diag(0,0,1,1,0,0,1)=-\frac{1}{6}|\tau|^2\Diag(0,0,1,1,0,0,1).
\end{align*}
As $\tau\wedge\tau=-2ae^{1256}$ and $\ast(\tau\wedge\tau)=-2ae^{347}$, it follows that $P=\spann\{e_7,e_3,e_4\}$, $Q=\spann\{e_1,e_2,e_5,e_6\}$ and therefore (iv) holds, concluding the proof.
\end{proof}

\section{Structure}\label{erp}

Our aim in this section is to discover and prove structural results for extremally Ricci-pinched $G_2$-structures on Lie groups.

Recall from Section \ref{closed} that the Lie groups endowed with a $G_2$-structure of the form $(G_\mu,\vp)$ (see Definition \ref{def-GmuG2}) cover the whole closed case up to equivariant equivalence (see Proposition \ref{closed-mu}).   The Lie algebra of $G_\mu$ decomposes as $\ggo=\RR e_7\oplus\hg$, where $\hg=\spann\{ e_1,\dots,e_6\}$ is a unimodular ideal, and $\vp$ is always given as in \eqref{phiA}.

The following proposition shows that under the ERP condition, the torsion $2$-form can be diagonalized in a very convenient way relative to the Lie algebra structure, which is certainly the starting point toward the structure results we will obtain in this section.

\begin{proposition}\label{ERP-mu}
Any Lie group endowed with a left-invariant ERP $G_2$-structure is equivariantly equivalent to $(G_\mu,\vp)$, up to scaling, for some $\mu=\lambda+\mu_A$ with $(\hg,\lambda)$ unimodular and $\tau_\mu=e^{12}-e^{56}$.
\end{proposition}

\begin{remark}
The $\SU(3)$-structure $(\omega,\rho^+)$ on the Lie algebra $\hg$ is therefore {\it half-flat}, i.e.\ $d_\lambda\omega\wedge\omega=0$ and $d_\lambda\rho^+ =0$ (see Proposition \ref{tauLA} and \eqref{muclosed}).
\end{remark}

\begin{proof}
Let $(G,\vp)$ be a Lie group endowed with a left-invariant ERP $G_2$-structure.  Consider the basis $\{ e_1,\dots,e_7\}$ of the Lie algebra $\ggo$ of $G$
such that $\vp$ has the form \eqref{phiA}.  Since the torsion form of $(G,\vp)$ $\tau$ belongs to $\Lambda^2_{14}\ggo^*$, it follows from
\eqref{tau-normal} and Proposition \ref{erp-prop}, (i) that it can be assumed to be given by $\tau= e^{12}-e^{56}$ (up to scaling).  As a first
consequence, $de^{347}=0$ by Proposition \ref{erp-prop}, (iii) and so if $c_{ijk}:=\la [e_i,e_j],e_k\ra$, then
\begin{align*}
0 =& de^{347} = de^3\wedge e^{47} - e^3\wedge de^4\wedge e^7 + e^{34}\wedge de^7 \\
=& -\sum_{i=1,2,5,6} c_{i33}e^{i347} -\sum_{i=1,2,5,6} c_{i44}e^{i347} -\sum_{i=1,2,5,6} c_{i77}e^{i347} \\
=& -\sum_{i=1,2,5,6} \tr{\ad{e_i}|_{\ggo_0}}e^{i347},
\end{align*}
which implies that
\begin{equation}\label{trg0}
\tr{\ad{e_i}|_{\ggo_0}}=0, \qquad\forall i=1,2,5,6,
\end{equation}
where $\ggo_0:=\spann\{ e_3,e_4,e_7\}$ is a Lie subalgebra by Proposition \ref{erp-prop}, (ii).  On the other hand, $\ggo_1:=\spann\{ e_1,e_2,e_5,e_6\}$ is also a subalgebra (see Proposition \ref{erp-prop}, (iii)) and hence by using Lemma \ref{tech3}, (i), we obtain that
$$
d\tau = \left(\theta(\ad{e_3}|_{\ggo_1})\tau\right)\wedge e^3 + \left(\theta(\ad{e_4}|_{\ggo_1})\tau\right)\wedge e^4 + \left(\theta(\ad{e_7}|_{\ggo_1})\tau\right)\wedge e^7 + d_{\ggo_1}\tau,
$$
where $d_{\ggo_1}:\Lambda^k\ggo_1^* \rightarrow \Lambda^{k+1}\ggo_1^*$ denotes the exterior derivative of $\ggo_1$.  But the ERP condition on $(G,\vp)$ reads $d\tau = \frac{1}{3}\vp - \frac{1}{3}e^{347}$, so it follows from \eqref{phi-intro} that $d_{\ggo_1}\tau=0$ and
$$
\theta(\ad{e_3}|_{\ggo_1})\tau =  \frac{1}{3}\omega_3, \qquad  \theta(\ad{e_4}|_{\ggo_1})\tau =  \frac{1}{3}\omega_4, \qquad \theta(\ad{e_7}|_{\ggo_1})\tau =  \frac{1}{3}\omega_7.
$$
This implies that the $2$-forms $\tau,\omega_3,\omega_4,\omega_7$ are all closed on the $4$-dimensional Lie algebra $\ggo_1$ by using that the maps $\ad{e_i}|_{\ggo_1}$ are derivations of $\ggo_1$ (see Lemma \ref{tech3}, (iii)), from which it is easy to see with the help of a computer that $\ggo_1$ is abelian.  From this and \eqref{trg0}, we obtain that $\ggo_1$ is contained in the ideal $\ug$ of $\ggo$ given by $\ug:=\{ X\in\ggo:\tr{\ad{X}}=0\}$.  If $G$ is not unimodular, then $\ggo=\RR X_0\oplus\ug$ is an orthogonal decomposition for some $X_0\in\la e_3,e_4,e_7\ra$, $|X_0|=1$.  It follows that there exists $h$ in the group $U_{\ggo_1,\tau}$ given in \eqref{Ug} such that $h(X_0)=e_7$.  The map $h$ therefore defines an equivariant equivalence between $(G,\vp)$ and $(G_\mu,\vp)$, where $\mu:=h\cdot\lb$, and we have that $h(\ug)=\hg$ (since $h$ is orthogonal) and $\tau_\mu=h\cdot\tau=\tau$.

In the case when $G$ is unimodular, it is proved in \cite[Theorem 6.7]{FinRff2} that $\ggo$ must be isomorphic to certain solvable Lie algebra.  In the present proof, we only use that $\ggo$ is solvable and we argue as in the beginning of the proof of \cite[Theorem 6.7]{FinRff2}.  Recall from Proposition \ref{erp-prop}, (iv) that $\Ricci\leq 0$ and the kernel of $\Ricci$ is $\ggo_1$.  The nilradical $\ngo$ of $\ggo$ is therefore contained in $\ggo_1$ by \cite[Lemma 1]{Dtt} and since $\ggo$ is solvable, $[\ggo,\ggo]\subset\ngo$.  Hence $\hg$ is an ideal of $\ggo$, concluding the proof.
\end{proof}

The following example shows that the above proposition is not valid in general for closed $G_2$-structures.

\begin{example}
Consider $(G_{\mu},\vp)$ with $\lambda(e_1,e_2)=e_3$, $\lambda(e_2,e_3)=4e_5$ and
$$
A=\left[\begin{smallmatrix}
1 & 0 & 0 & 0 & 0 & 0 \\ 0 & -1 & 0 & 0 & 0 & 0 \\ 0 & 0 & 0 & 0 & 0 & 0 \\ 0 & 0 & 0 & 2 & 0 & 0  \\ 0 & 0 & 0 & 0 & -1 & 0 \\ 0 & 1 & 0 & 0 & 0 & -3
\end{smallmatrix}\right],
$$
with respect to the basis $\{e_1,\dots,e_6\}$. It is straightforward to check that $d_\mu\vp=0$ and  $\tau_\mu=-2e^{12}-e^{16}-4e^{34}-e^{37}+6e^{56}$.
Since $\hg$ is the nilradical of $\mu$, the torsion $2$-form $\tau_{\mu_1}$ of any $(G_{\mu_1},\vp)$ equivariantly equivalent to $(G_\mu,\vp)$ will satisfy
that $\tau_{\mu_1}(e_7,\cdot)$ is not identically zero.   Indeed, any orthogonal isomorphism between $G_\mu$ and $G_{\mu_1}$ must stabilize both $\hg$ and
$\RR e_7$.
\end{example}

The diagonalization of $\tau$ obtained in Proposition \ref{ERP-mu} makes the equivalence problem much simpler to tackle.
Recall the subgroups $U_{\hg,\tau}$ and $U_{\ggo_1,\tau}$ of $G_2$ described in Section \ref{subg-sec}.

\begin{proposition}\label{equiv}
Assume that $(G_{\mu_1},\vp)$ and $(G_{\mu_2},\vp)$ have $\tau_{\mu_1}=\tau_{\mu_2}=e^{12}-e^{56}$.  Then they are equivariantly equivalent if and only if $\mu_2=h\cdot\mu_1$ for some $h\in U\subset G_2$, where
\begin{itemize}
\item[(i)] $U=U_{\hg,\tau}$ if they are not unimodular; and
\item[(ii)] $U=U_{\ggo_1,\tau}$ if they are unimodular and $\ggo_1=\spann\{ e_1,e_2,e_5,e_6\}$ is their nilradical.
\end{itemize}
\end{proposition}

\begin{proof}
In the non-unimodular case (i), $\hg$ is a characteristic ideal of both Lie algebras by Proposition \ref{closed-mu} and so any equivariant equivalence $h$ between them must leave $\hg$ invariant and stabilize $\tau$, that is, $h\in U_{\hg,\tau}$.  On the other hand, part (ii) follows from the fact that $h$ must leave $\ggo_1$ invariant (i.e.\ $h\in U_{\ggo_1}$) being $\ggo_1$ the nilradical of both Lie algebras, and so $h\in U_{\ggo_1,\tau}$ since $h\cdot\tau=\tau$  (see Section \ref{subg-sec}).

The converse easily follows from the fact that $U\subset G_2$.
\end{proof}

In the light of Proposition \ref{ERP-mu}, we consider from now on a closed $G_2$-structure $(G_\mu,\vp)$ such that
\begin{equation}\label{tau-linda}
\tau:=\tau_\mu=e^{12}-e^{56}.
\end{equation}
In that case,
\begin{equation}\label{linda-1}
d_\lambda\omega\wedge\omega=0,
\end{equation}
by Proposition \ref{tauLA}, $\tau\wedge\tau=-2e^{1256}$ and $\ast(\tau\wedge\tau)=-2e^{347}$.  This implies that $(G_\mu,\vp)$ is ERP if and only if
\begin{equation}\label{erp-linda}
d_\mu\tau = \frac{1}{3}\vp - \frac{1}{3}e^{347},
\end{equation}
which is equivalent by Lemma \ref{tech3}, (i) to
 \begin{equation}\label{erp-linda-2}
d_\lambda\tau = \frac{1}{3}\rho^+, \qquad \theta(A)\tau = \frac{1}{3}(e^{12}+e^{56}).
\end{equation}
It follows from \eqref{muclosed} and Lemma \ref{tech3}, (iii) that
\begin{align*}
d_\lambda\omega &= d_\lambda(e^{12}+e^{56}) + d_\lambda e^{34} = 3d_\lambda\theta(A)\tau + d_\lambda e^{34} \\
&= 3\theta(A)d_\lambda\tau + d_\lambda e^{34} = \theta(A)\rho^+ + d_\lambda e^{34} =d_\lambda\omega+d_\lambda e^{34},
\end{align*}
and consequently,
\begin{equation}\label{d34}
d_\lambda e^{34} = 0.
\end{equation}

Some algebraic and geometric consequences of Proposition \ref{erp-prop} follow.

\begin{proposition}\label{erp-prop3}
If $(G_\mu,\vp)$ is ERP with $\tau=e^{12}-e^{56}$, then,
\begin{itemize}
\item[(i)] $\ggo_0:=\spann\{ e_7,e_3,e_4\}$,  $\ggo_1:=\spann\{ e_1,e_2,e_5,e_6\}$ and $\hg_1:=\spann\{ e_3,e_4\}$ are Lie subalgebras of $\ggo$.

\item[(ii)] The Ricci operator $\Ricci_\mu$ of $(G_\mu,\ip)$ is diagonal with respect to $\{ e_i\}$ and $\Ricci_\mu|_{\ggo_0}=-\frac{1}{3}I$, $\Ricci_\mu|_{\ggo_1}=0$.

\item[(iii)] If $Q_\mu$ is the unique symmetric operator of $\ggo$ such that $\theta(Q_\mu)\vp=d_\mu\tau$, then
$$
\Ricci_\mu=-\frac{1}{3}I-2Q_\mu; \qquad\mbox{in particular}, \quad Q_\mu|_{\ggo_0}=0, \quad Q_\mu|_{\ggo_1}=-\frac{1}{6}I.
$$
\end{itemize}
\end{proposition}

\begin{remark}\label{steady}
It follows from part (iii) that $Q_\mu\in\Der(\ggo)$, and in particular $(G_\mu,\vp)$ is a steady Laplacian soliton (see \cite[Theorem 3.8]{LF}) and $(G_\mu,\ip)$ is an expanding Ricci soliton (see \cite[(5)]{solvsolitons}), if and only if $\ggo_1$ is an abelian ideal of $\ggo$.
\end{remark}

\begin{proof}
It is well known that the kernel of any closed $k$-form on a Lie algebra is a Lie subalgebra.  Since $\tau\wedge\tau=-2e^{1256}$ and $\ast(\tau\wedge\tau)=-2e^{347}$, it follows from Proposition \ref{erp-prop} (ii), (iii) that  $\ggo_0$ and $\ggo_1$ are Lie subalgebras of $\ggo$.  In particular, $\hg_1=\hg\cap\ggo_0$ is also a subalgebra.  Parts (ii) and (iii) are direct consequences of \cite[(15)]{LS-ERP} (for $q=\tfrac{1}{6}$) and \cite[(12)]{LS-ERP}.
\end{proof}

We will now show that the ERP condition actually imposes much stronger constraints on the structure of the Lie algebra.  Let us first introduce some notation.  Consider
$$
\spg(\ggo_1,\tau) := \left\{ E\in\glg(\ggo_1):-\theta(E)\tau = \tau(E\cdot,\cdot)+\tau(\cdot,E\cdot)=0\right\}, \quad \tau=e^{12}-e^{56},
$$
and note that $E\in\spg(\ggo_1,\tau)$ if and only if written in terms of the basis $\{ e_1,e_2,e_5,e_6\}$,
\begin{equation}\label{spt}
E=\left[\begin{smallmatrix}
E_{11}&E_{12}&E_{15}&E_{16} \\ E_{21}&-E_{11}&E_{25}&E_{26} \\ E_{26}&-E_{16}&E_{55}&E_{56} \\ -E_{25}&E_{15}&E_{65}&-E_{55}
\end{smallmatrix}\right].
\end{equation}
We also consider the following three matrices,
\begin{equation}\label{erp-st-cond}
T_7:=\left[\begin{smallmatrix}
-\frac{1}{3} &&& \\ &0&& \\  &&\frac{1}{3}&\\ &&&0
\end{smallmatrix}\right], \quad
T_3:=\left[\begin{smallmatrix}
&&&\frac{1}{3} \\ &&0& \\  &0&&\\ \frac{1}{3}&&&
\end{smallmatrix}\right], \quad
T_4:=\left[\begin{smallmatrix}
&&0& \\ &&&-\frac{1}{3} \\  0&&&\\ &-\frac{1}{3}&&
\end{smallmatrix}\right].
\end{equation}
for which it is easy to check that
\begin{equation}\label{titaT}
\theta(T_7)\tau=\frac{1}{3}\omega_7, \qquad \theta(T_3)\tau=\frac{1}{3}\omega_3, \qquad \theta(T_4)\tau=\frac{1}{3}\omega_4.
\end{equation}

The following is our main structural result.  Recall from Proposition \ref{ERP-mu} that any left-invariant ERP $G_2$-structure on a Lie group is
equivariantly equivalent to some $(G_\mu,\vp)$ with $\tau=e^{12}-e^{56}$ and $\hg=\spann\{e_1,\dots,e_6\}$ unimodular ($\ggo=\RR e_7\oplus\hg$).

\begin{theorem}\label{erp-prop2}
Let $(G_\mu,\vp)$ be an ERP $G_2$-structure with $\tau=e^{12}-e^{56}$ and $\hg$ unimodular.  Then, the following conditions hold:
\begin{itemize}
\item[(i)] $\ggo_0=\spann\{ e_7,e_3,e_4\}$ is a Lie subalgebra  and $\ggo_1=\spann\{ e_1,e_2,e_5,e_6\}$ is an abelian ideal of $\ggo$.  In particular, $\ggo=\ggo_0\ltimes\ggo_1$ and $\ggo$ is solvable.

\item[(ii)] $\hg_1=\spann\{ e_3,e_4\}$ is an abelian subalgebra (so $\hg=\hg_1\ltimes\ggo_1$).

\item[(iii)] There exist $E,F,G\in\spg(\ggo_1,\tau)$ such that
$$
A_2=E+T_7, \quad B_2=F+T_3, \quad C_2=G+T_4,
$$
where $A_1:=A|_{\hg_1}$, $A_2:=A|_{\ggo_1}$, $B_2:=\ad{e_3}|_{\ggo_1}$ and $C_2:=\ad{e_4}|_{\ggo_1}$.  In particular, $\tr{A_2}=\tr{B_2}=\tr{C_2}=0$ and $[B_2,C_2]=0$.
\end{itemize}
\end{theorem}

\begin{proof}
We first prove part (iii).  Recall that $\ggo_0$, $\ggo_1$ and $\hg_1$ are all Lie subalgebras of $\ggo$.   It was shown in the proof of Proposition \ref{ERP-mu} that,
$$
\theta(A_2)\tau=\frac{1}{3}\omega_7, \qquad \theta(B_2)\tau=\frac{1}{3}\omega_3, \qquad \theta(C_2)\tau=\frac{1}{3}\omega_4,
$$
and  thus $A_2-T_7$, $B_2-T_3$ and $C_2-T_4$ all belong to $\spg(\ggo_1,\tau)$ and the first assertion in part (iii) follows.  Note that  $\tr{A_2}=\tr{B_2}=\tr{C_2}=0$ and so $\lambda(e_3,e_4)=0$ (i.e.\ $\hg_1$ abelian) follows from the fact that $\hg$ is unimodular, completing the proof of parts (ii) and (iii).

It was also obtained in the proof of Proposition \ref{ERP-mu} that $\ggo_1$ is abelian.  We now prove that $\ggo_1$ is an ideal, which will conclude the proof of the theorem.  If we set $B:=\ad{e_3}|_{\hg}$ and $C:=\ad{e_4}|_{\hg}$, then from \eqref{muclosed},
\begin{align*}
0=&d_\lambda \rho^+=d_\lambda\omega_3\wedge e^3 + \omega_3 \wedge d_\lambda e^3+d_\lambda \omega_4\wedge e^4+\omega_4\wedge d_\lambda e^4\\
 =&d_{\ggo_1}\omega_3\wedge e^3 - \theta(C_2)\omega_3\wedge e^{34} - \omega_3 \wedge \theta(B) e^3\wedge e^3 - \omega_3 \wedge \theta(C) e^3\wedge e^4\\
 &+d_{\ggo_1} \omega_4\wedge e^4+\theta(B_2)\omega_4\wedge e^{34} - \omega_4\wedge  \theta(B) e^4\wedge e^3 -\omega_4\wedge \theta(C) e^4\wedge e^4\\
 =&\left(- \theta(C_2)\omega_3+\theta(B_2)\omega_4\right)\wedge e^{34}  - \left(\omega_3 \wedge \theta(B) e^3 +\omega_4\wedge \theta(B) e^4\right)\wedge e^3\\
 & -\left(\omega_3 \wedge \theta(C) e^3+\omega_4\wedge \theta(C) e^4\right)\wedge e^4.
\end{align*}
Since $\theta(B) e^3,\theta(B) e^4,\theta(C) e^3,\theta(C) e^4 \in\Lambda^1\ggo_1^*$, it follows that
\begin{align*}
0=&\omega_3 \wedge \theta(B) e^3 +\omega_4\wedge \theta(B) e^4=\sum_{1,2,5,6} \left(\omega_3 \wedge c_{i33}e^i +\omega_4\wedge c_{i34}e^i\right)\\
=&(c_{133}-c_{234})e^{126}+(c_{233}+c_{134})e^{125}+(c_{353}-c_{364})e^{256}+(c_{363}+c_{354})e^{156},
\end{align*}
and
\begin{align*}
0=& \omega_3 \wedge \theta(C) e^3+\omega_4\wedge \theta(C) e^4= \sum\left(\omega_3 \wedge c_{i43}e^i +\omega_4\wedge c_{i44}e^i\right)\\
=& (c_{143}-c_{244}e)e^{126}+(c_{243}+c_{144})e^{125}+(c_{453}-c_{464})e^{256}+(c_{463}+c_{454})e^{156}.
\end{align*}
Moreover,
$$
0=  d_\lambda e^{34} = d_\lambda e^3\wedge e^4 - e^3 \wedge d_\lambda e^{4} =-\left( \theta(B) e^3 + \theta(C) e^{4}\right) \wedge e^{34}=\sum_{1,2,5,6}(c_{3i3}+c_{4i4})e^{i34}.
$$
Summarizing, we have obtained that
\begin{align}
&c_{133}=-c_{144}=c_{234}=c_{243}, \qquad  c_{353}=c_{364}=-c_{454}=c_{463},\label{coso} \\
&c_{134}=c_{143}=-c_{233}=c_{244}, \qquad  c_{354}=-c_{363}=c_{453}=c_{464}.\nonumber
\end{align}
As before,
\begin{align*}
0=&\la [\ad{e_1},\ad{e_2}]e_4,e_3\ra=\la \ad{e_1}\ad{e_2}e_4-\ad{e_2}\ad{e_1}e_4,e_3\ra\\
=&\sum_{i=1}^7\la c_{24i}\ad{e_1}(e_i)-c_{14i}\ad{e_2}(e_i),e_3\ra=\sum_{i,j=1}^7\la c_{24i}c_{1ij}e_j-c_{14i}c_{2ij}e_j,e_3\ra\\\
=&\sum_{i=1}^7( c_{24i}c_{1i3}-c_{14i}c_{2i3})=c_{243}c_{133}+ c_{244}c_{143}-c_{143}c_{233}-c_{144}c_{243}\\
=& 2(c_{133}^2+ c_{134}^2).
\end{align*}
In much the same way, one obtains that $ 0=\la [\ad{e_5},\ad{e_6}](e_4),e_3\ra= 2(c_{353}^2+ c_{354}^2)$. Thus, $c_{133}=c_{134}=c_{353}=c_{354}=0$ and so
it follows from \eqref{coso} that $[\ggo_1,\hg]\subset\ggo_1$.

Therefore, it only remains to show that  $[\ggo_1,e_7]\subset \ggo_1$.  It follows from
$\tau=e^{12}-e^{56}$, $\hg$ unimodular and $\ggo_1$ abelian that
\begin{align*}
0=&\la e^{13},\tau\ra\vol= e^{13} \wedge \ast\tau = -e^{13} \wedge d\ast\vp \\
=& -d(e^{13} \wedge \ast\vp)+ de^{13} \wedge \ast\vp
 =-d(e^{123467})+ \la de^{13},\vp\ra\vol \\
 =& \tr(\ad{e_5})\vol + \la de^1 \wedge e^3,\vp\ra \vol + \la e^1 \wedge de^3,\vp\ra\vol =-c_{273}.
\end{align*}
In the same manner, we can see that $0=c_{i7j}$ for each $i\in\{1,2,5,6\}$ and $j\in\{3,4,7\}$.  This implies that $\la [\ggo_1,e_7],\ggo_0\ra$ vanishes and so $\ggo_1$ is an ideal, as desired.
\end{proof}

The following geometric consequence of Theorem \ref{erp-prop2} follows from Remark \ref{steady}.

\begin{corollary}
Any left-invariant ERP $G_2$-structure on a Lie group is both a steady Laplacian soliton and an expanding Ricci soliton.
\end{corollary}

Recall that all the examples of Laplacian solitons found in \cite{LF, Ncl} are expanding.

We now give the converse of Theorem \ref{erp-prop2}, which paves the way to the search for examples and eventually, to a full classification.   In addition to \eqref{phi-intro}, we denote by
$$
\overline{\omega_3}:=e^{26}+e^{15}, \qquad \overline{\omega_4}:=e^{16}-e^{25}.
$$

\begin{proposition}\label{erp-iff}
Let $\mu$ denote a Lie bracket on $\ggo$ whose only nonzero structural constants are given by $A_1$, $A_2$, $B_2$ and $C_2$ as in Theorem \ref{erp-prop2}.   Then $(G_\mu,\vp)$ is ERP with $\tau_\mu=e^{12}-e^{56}$ if and only if there exist $E,F,G\in\spg(\ggo_1,e^{12}-e^{56})$ (see \eqref{spt}) such that the following conditions hold:
\begin{itemize}
\item[(i)] $A_2=E+T_7$, $B_2=F+T_3$ and $C_2=G+T_4$, where the $T_i$'s are defined as in \eqref{erp-st-cond}.
\item[ ]
\item[(ii)] $\theta(E^t)\omega_7+\theta(F^t)\omega_3+\theta(G^t)\omega_4 = -(\tr{A_1})\omega_7$.
\end{itemize}
\end{proposition}

\begin{remark}
The Jacobi condition for such a $\mu$ is equivalent to
\begin{equation}\label{Jac}
 [A_2,B_2]=aB_2+cC_2, \quad [A_2,C_2]=bB_2+dC_2, \quad [B_2,C_2]=0, \quad\mbox{where}\;
 A_1=\left[\begin{smallmatrix} a&b\\ c&d\end{smallmatrix}\right].
\end{equation}
\end{remark}

{\small \begin{table}
\renewcommand{\arraystretch}{1.6}
$$
\begin{array}{|c||c|c|c|c|c|c|}\hline
   & \tau & \omega_7 & \omega_3 & \omega_4 & \overline{\omega_3} & \overline{\omega_4}\\
\hline\hline
T_7 & \tfrac{1}{3} \omega_7 & \tfrac{1}{3} \tau & 0 & \tfrac{1}{3} \bar{\omega}_4 & 0 & -\tfrac{1}{3} \omega_4 \\
\hline
T_3 & \tfrac{1}{3} \omega_3 & \tfrac{1}{3} \overline{\omega_3} & \tfrac{1}{3} \tau & 0 & -\tfrac{1}{3} \omega_7 & 0 \\
\hline
T_4 & \tfrac{1}{3} \omega_4 &  \tfrac{1}{3} \overline{\omega_4} & 0 &  \tfrac{1}{3} \tau & 0 &  -\tfrac{1}{3} \omega_7  \\
\hline
\end{array}
$$
\caption{$T_i$-actions on $2$-forms}\label{Tact}
\end{table}}

\begin{proof}
We first suppose that $(G_\mu,\vp)$ is ERP with $\tau_\mu=e^{12}-e^{56}$.  Part (i) follows from Theorem \ref{erp-prop2}.  In order to prove (ii), we now proceed to compute $\tau_\mu$ by using the formula given in Proposition \ref{tauLA} and Table \ref{Tact} (recall from \eqref{linda-1} that $d_\lambda\omega\wedge\omega=0$):
\begin{align*}
-\ast_\hg d_\lambda \rho^-
    =&-\ast_\hg  (e^3\wedge d_\lambda\omega_4 - e^4\wedge d_\lambda\omega_3)
    =-\ast_\hg  (e^{34} \wedge (\theta(C_2)\omega_4+\theta(B_2)\omega_3)) \\
    =&-\ast_{\ggo_1}\theta(C_2)\omega_4-\ast_{\ggo_1}\theta(B_2)\omega_3
    =\theta(C_2^t)\ast_{\ggo_1}\omega_4+\theta(B_2^t)\ast_{\ggo_1}\omega_3\\
    =&\theta(C_2^t)\omega_4+\theta(B_2^t)\omega_3
    =\theta(G^t)\omega_4+\theta(T_4)\omega_4
    +\theta(F^t)\omega_3+\theta(T_3)\omega_3\\
    =&\theta(G^t)\omega_4 + \theta(F^t)\omega_3 + \frac{2}{3}\left(e^{12}-e^{56}\right),
\end{align*}
and on the other hand,
\begin{align*}
(\tr{A})\omega+\theta(A^t)\omega=&(\tr{A_1})e^{34}+(\tr{A_1})\omega_7+\theta(A_2^t)\omega_7+\theta(A_1^t)e^{34}\\
    =&(\tr{A_1})\omega_7+\theta(E^t)\omega_7+\frac{1}{3}(e^{12}-e^{56}).
\end{align*}
Thus part (ii) follows from the fact that $\tau_\mu=e^{12}-e^{56}$.

Conversely, assume that parts (i) and (ii) hold.   Using part (i), \eqref{muclosed} and Table \ref{Tact} , it is easy to see that $d_\mu\vp=0$ if and only if
\begin{align}
\theta(F)\omega_7+a\omega_3+c\omega_4 =& \theta(E)\omega_3 - \frac{1}{3}\overline{\omega_3}, \notag \\
\theta(G)\omega_7+b\omega_3+d\omega_4 =& \theta(E)\omega_4, \label{clo}\\
\theta(F)\omega_4=&\theta(G)\omega_3. \notag
\end{align}
But straightforwardly, one obtains that these equalities respectively follow by just evaluating $\theta([A_2,B_2])$, $\theta([A_2,C_2])$ and $\theta([B_2,C_2])$ at $\tau$ and using the Jacobi condition \eqref{Jac}.   On the other hand, since
\begin{align*}
d_\lambda\omega\wedge\omega =& \unm d_\lambda(\omega\wedge\omega)=d_\lambda(e^{1234}+e^{3456}+e^{1256})=d_\lambda(e^{1256})\\
                              =& \theta(B_2)e^{1256}\wedge e^3+ \theta(C_2)e^{1256}\wedge e^4=-\tr{B_2}e^{12356}-\tr{C_2}e^{12456}=0,
\end{align*}
we obtain from part (ii) that $\tau_\mu=e^{12}-e^{56}$.  It now follows from \eqref{erp-linda-2} and part (i) that $(G_\mu,\vp)$ is ERP, which concludes the proof of the proposition.
\end{proof}

The strong conditions on the Ricci curvature imposed by ERP (see Proposition \ref{erp-prop3}, (iii)) produce very useful constraints on the matrices involved.

\begin{proposition}\label{erp-ricci}
If $(G_\mu,\vp)$ is ERP with $\tau_\mu=e^{12}-e^{56}$, say $\mu=(A_1,A_2,B_2,C_2)$, then the following conditions hold:
\begin{itemize}
\item[(i)] $\tr{S(A_1)^2}+\tr{S(A_2)^2} = \frac{1}{3}$.
\item[ ]
\item[(ii)] $\unm[A_2,A_2^t] + \unm[B_2,B_2^t] + \unm[C_2,C_2^t] = (\tr{A_1})S(A_2)$.
\item[ ]
\item[(iii)] $\tr{S(A_2)S(B_2)}=\tr{S(A_2)S(C_2)}=0$.
\item[ ]
\item[(iv)]  $\left[\begin{smallmatrix} \tr{S(B_2)^2}&\tr{S(B_2)S(C_2)} \\ \tr{S(B_2)S(C_2)}&\tr{S(C_2)^2}\end{smallmatrix}\right]
- \unm [A_1,A_1^t] + (\tr{A_1})S(A_1) = \left[\begin{smallmatrix}
 \frac{1}{3}&0 \\ 0& \frac{1}{3}\end{smallmatrix}\right] $
\end{itemize}
\end{proposition}

\begin{proof}
All the items follow from Proposition \ref{erp-prop3}, (ii) by just applying the formula for the Ricci operator of a solvmanifold given in \cite[(25)]{solvsolitons}.
\end{proof}

We also note that if $(G_\mu,\vp)$ is ERP with $\tau_\mu=e^{12}-e^{56}$, then $(G_\mu,\ip)$ is a solvsoliton;  indeed, in terms of the decomposition $\ggo=\ggo_0\oplus\ggo_1$,
$$
\Ricci_\mu = -\frac{1}{3}I +  \left[\begin{array}{cc} 0& \\ & \frac{1}{3}I\end{array}\right] \in\RR I + \Der(\mu).
$$
This allows us to use, in addition to Proposition \ref{erp-ricci}, the structure theorem for solvsolitons \cite[Theorem 4.8]{solvsolitons}.

\section{Examples and structure refinements}\label{cla-sec}

Acording to Theorem \ref{erp-prop2}, for any ERP $(G_\mu,\vp)$ with $\tau=e^{12}-e^{56}$, $\ggo_1=\spann\{ e_1,e_2,e_5,e_6\}$ is an abelian ideal of the Lie algebra $(\ggo,\mu)$.  Thus the nilradical $\ngo$ of $(\ggo,\mu)$ contains $\ggo_1$ and so $\dim{\ngo}\geq 4$.  Recall from Proposition \ref{erp-iff} that the Lie bracket has always the form $\mu=(A_1,A,B,C)$ for certain matrices $A_1\in\glg_2(\RR)$ and $A,B,C\in\glg_4(\RR)$ such that $[B,C]=0$.

We can use Proposition \ref{equiv} to consider the equivalence problem.  The action of the group $U_{\hg,\tau}$ on $\mu=(A_1,A,B,C)$ can be described as follows (see Section \ref{subg-sec}).  If $h\in U_0$, say with $h_1=\left[\begin{smallmatrix}
x&y \\ -y&x\end{smallmatrix}\right]$, $x^2+y^2=1$ and $h_2:=\left[\begin{array}{cc}
h_3&0 \\ 0&h_4\end{array}\right]$, $h_3,h_4\in\SO(2)$, then
\begin{equation}\label{equiv1}
h\cdot\mu = \left(h_1A_1h_1^{-1}, h_2Ah_2^{-1},  h_2(xB-yC)h_2^{-1}, h_2(yB+xC)h_2^{-1}\right),
\end{equation}
and if $g_1:=\left[\begin{smallmatrix}
1& \\ &-1\end{smallmatrix}\right]$ and $g_2:=\left[\begin{smallmatrix}
0&0&1&0\\ 0&0&0&-1\\ -1&0&0&0\\ 0&1&0&0\end{smallmatrix}\right]$, then
\begin{equation}\label{equiv2}
g\cdot\mu = \left(-g_1A_1g_1^{-1}, -g_2Ag_2^{-1}, g_2Bg_2^{-1}, -g_2Cg_2^{-1}\right).
\end{equation}

Let $(G_\mu,\vp)$ be an ERP $G_2$-structure with $\tau=e^{12}-e^{56}$ and nilradical $\ngo$, say $\mu=(A_1,A,B,C)$.  If $\mu$ is unimodular, then $\ngo=\ggo_1$ (see Proposition \ref{erp-unim} below) and in the non-unimodular case, $\ggo_1\subset\ngo\subset\hg$.  In any case, $A_1$ and $A$ are necessarily normal matrices by \cite[Theorem 4.8]{solvsolitons}.

In what follows, we separately study each of the cases $\dim{\ngo}=4,5,6$; note that $\mu$ can not be nilpotent since $\Ricci\leq 0$ (see \cite{Wlf,Mln}).

\subsection{Case $\dim{\ngo}=4$}
In the unimodular case, some necessary algebraic conditions proved by I. Dotti \cite{Dtt} for $\Ricci\leq 0$ give rise to the following characterization.

\begin{proposition}\label{erp-unim}
If $(G_\mu,\vp)$ is ERP with $\tau=e^{12}-e^{56}$, say $\mu=(A_1,A,B,C)$, then the following conditions are equivalent:
\begin{itemize}
\item[(i)] $\mu$ is unimodular (i.e.\ $\tr{A_1}=0$).
\item[(ii)] $A_1=0$ (in particular, $A,B,C$ pairwise commute).
\item[(iii)] $\ggo_1$ is the nilradical of $\mu$ (in particular, $\{ A,B,C\}$ is linearly independent).
\end{itemize}
\end{proposition}

\begin{proof}
Recall from Proposition \ref{erp-prop}, (iv) that $\Ricci\leq 0$ and the kernel of $\Ricci$ is $\ggo_1$.  If $\mu$ is unimodular, then the nilradical $\ngo$ of $\ggo$ is contained in $\ggo_1$ by \cite[Lemma 1]{Dtt}, but $\ggo_1\subset\ngo$ as $\ggo_1$ is an abelian ideal of $\ggo$, so $\ngo=\ggo_1$.  Since the image of any derivation of a solvable Lie algebra is contained in the nilradical, we obtain that $A_1=0$.  The remaining implications trivially hold.
\end{proof}

\begin{proposition}\label{erp-unim-2}
If $(G_\mu,\vp)$ is ERP with $\tau=e^{12}-e^{56}$ and $\mu$ is unimodular, say $\mu=(0,A,B,C)$, then the $4\times 4$ matrices $A,B,C$ are all symmetric, they pairwise commute and the set $\left\{\sqrt{3}A, \sqrt{3}B, \sqrt{3}C\right\}$ is orthonormal.
\end{proposition}

\begin{remark}\label{erp-unim-3}
In particular, $G_\mu$ is isomorphic to the Lie group given in \cite[Example 4.7]{LS-ERP} and Example \ref{J} below.  This has been proved in \cite[Theorem 6.7]{FinRff2}.  We note however that there could be other non-equivalent ERP $G_2$-structures on $G_\mu$.
\end{remark}

\begin{proof}
From equation (ii) in Proposition \ref{erp-ricci} (recall that $A_1=0$), we obtain that the matrices $A,B,C$ are all normal, by just multiplying with each of the three terms (alternatively, one can just apply \cite[Theorem 4.8]{solvsolitons}).  Thus $A,B,C,S(A),S(B),S(C)$ is a commuting family of normal $4\times 4$ matrices, which are all non-zero by Proposition \ref{erp-ricci}, (i) and (iv).  The only possibility for this to happen is that they are all symmetric, and so the set $\left\{\sqrt{3}A, \sqrt{3}B, \sqrt{3}C\right\}$ is orthonormal by Proposition \ref{erp-ricci}, (i), (iii) and (iv), as desired.
\end{proof}

\begin{example}\label{J}
Consider $\mu_J:=(0,A,B,C)$, where
$$
A=\left[\begin{smallmatrix}
-\frac{1}{6}&&& \\ &-\frac{1}{6}&& \\ &&\frac{1}{2}&\\ &&&-\frac{1}{6}
\end{smallmatrix}\right], \quad
B=\left[\begin{smallmatrix}
0&-\frac{\sqrt{2}}{6}&0&\frac{1}{3} \\ -\frac{\sqrt{2}}{6}&0&0&0\\ 0&0&0&0\\ \frac{1}{3}&0&0&0
\end{smallmatrix}\right], \quad
C=\left[\begin{smallmatrix}
\frac{\sqrt{2}}{6}&0&0&0 \\  0&-\frac{\sqrt{2}}{6}&0&-\frac{1}{3} \\ 0&0&0&0 \\ 0&-\frac{1}{3}&0&0
\end{smallmatrix}\right].
$$
It is straightforward to check that all the conditions in Proposition \ref{erp-iff} hold for these matrices, thus $(G_{\mu_J},\vp)$ is an ERP $G_2$-structure with $\tau=e^{12}-e^{56}$, and also that the map
$$
h:=\frac{1}{6}\left[\begin{smallmatrix}0 & 0 & 3\sqrt{2} & 3\sqrt{2} & 0 & 0 & 0 \\
0 & 0 & \sqrt{6} & -\sqrt{6} & -2\sqrt{6} & 0 & 0 \\
-3\sqrt{2}  & -3\sqrt{2} & 0 & 0 & 0 & 0 & 0 \\
-\sqrt{6} & \sqrt{6} & 0 & 0 & 0 & 0 & 2\sqrt{6} \\
0 & 0  & 0 & 0 & 0 & -6 & 0\\
0 & 0 & -2\sqrt{3} & 2\sqrt{3} & -2\sqrt{3} & 0 & 0 \\
-2\sqrt{3} & 2\sqrt{3} & 0 & 0 & 0 & 0 & -2\sqrt{3} \end{smallmatrix}\right]\in G_2
$$
defines an equivariant equivalence between $(G_{\mu_J},\vp)$ and \cite[Example 4.7]{LS-ERP}.
\end{example}

The difficulty in finding new examples in this case relies on the complicated structure of the $4$-dimensional group $U_{\ggo_1,\tau}$ (see \eqref{Ug}) providing the equivariant equivalence.

\subsection{Case $\dim{\ngo}=5$}
By acting with $U_{\hg,\tau}$ if necessary (see \eqref{equiv1}), we can assume in this case that up to equivariant equivalence, $\ngo=\RR e_4\oplus\ggo_1$.  Let $(G_\mu,\vp)$ be an ERP $G_2$-structure with $\tau=e^{12}-e^{56}$ and $\ngo$ as above, say $\mu=(A_1,A,B,C)$.  It follows from \cite[Theorem 4.8]{solvsolitons} that $A_1,A,B$ are normal and $[A,B]=0$, and since $[e_7,\ngo]\subset\ngo$, one further obtains that
$$
A_1=\left[\begin{smallmatrix} 0&0\\ 0&d\end{smallmatrix}\right], \quad d\ne 0, \qquad [A,C]=dC, \qquad [B,C]=0.
$$
By acting with $g$ if necessary as in \eqref{equiv2}, one can assume up to equivariant equivalence that $d>0$.

The following two Lie brackets provide new examples of ERP $G_2$-structures $(G_\mu,\vp)$ with $\tau=e^{12}-e^{56}$ and $\ngo=\RR e_4\oplus\ggo_1$ by Proposition \ref{erp-iff}.

\begin{example}\label{M2}
Consider $\mu_{M2}:=(A_1,A,B,C)$, where
$$
A_1=\left[\begin{smallmatrix} 0&\\ &\frac{1}{3}\end{smallmatrix}\right], \quad
A=\left[\begin{smallmatrix}
-\frac{1}{3} &  &  &  \\
  & 0 & &  \\
  &  & 0 &  \\
  &  & & \frac{1}{3}
\end{smallmatrix}\right], \quad
B=\left[\begin{smallmatrix}
-\frac{1}{6} & 0 & 0 & 0 \\
0  & \frac{1}{6} & \frac{1}{3} & 0 \\
0  & \frac{1}{3} & \frac{1}{6} & 0 \\
0  & 0 & 0 & -\frac{1}{6} \\
\end{smallmatrix}\right], \quad
 C=\left[\begin{smallmatrix}
0 & & &  \\
-\frac{1}{3} & 0 &  &  \\
\frac{1}{3} & 0 & 0 &  \\
0  & -\frac{1}{3} & \frac{1}{3} & 0 \\
\end{smallmatrix}\right].
$$
Note that the nilradical $\ngo$ is $3$-step nilpotent.
\end{example}

\begin{example}\label{M3}
Consider $\mu_{M3}:=(A_1,A,B,C)$, where
$$
A_1=\frac{1}{6}\left[\begin{smallmatrix} 0&0\\0 & \sqrt{6} \end{smallmatrix}\right], \quad
A=\frac{1}{12}\left[\begin{smallmatrix}
-2 & 0 & -\sqrt{2} & 0 \\
0 & -2 & 0 & -\sqrt{2} \\
-\sqrt{2} & 0 & 2 & 0 \\
0 & -\sqrt{2} & 0 & 2
\end{smallmatrix}\right],
$$
$$
B=\frac{1}{6}\left[\begin{smallmatrix}
0 & \sqrt{2} & 0 & 1 \\
\sqrt{2} & 0 & 1 & 0 \\
0 & 1 & 0 & -\sqrt{2} \\
1 & 0 & -\sqrt{2} & 0
\end{smallmatrix}\right], \quad
C=\frac{1}{12}\left[\begin{smallmatrix}
-\sqrt{2} & 0 & 2-\sqrt{6} & 0\\
0 & \sqrt{2} & 0 & -2+\sqrt{6}\\
2+\sqrt{6} & 0 & \sqrt{2} & 0 \\
0 & -2-\sqrt{6} & 0 & -\sqrt{2}
\end{smallmatrix}\right].
$$
The nilradical $\ngo$ is $2$-step nilpotent in this case.
\end{example}

By considering the possible forms for the normal matrices $A$ and $B$ under the condition given in Proposition \ref{erp-iff}, (i), it can be shown with some computer assistance that $[A,B]=0$ never holds unless $A$ and $B$ are both symmetric.

\subsection{Case $\dim{\ngo}=6$}
We have that $\ngo=\hg$ in this case, so $B$ and $C$ are nilpotent.  Let $(G_\mu,\vp)$ be an ERP $G_2$-structure with $\tau=e^{12}-e^{56}$ and nilradical $\ngo=\hg$, say $\mu=(A_1,A,B,C)$.   By using \eqref{equiv1}, we can assume that up to equivariant equivalence,
\begin{itemize}
\item[(i)] either $A_1=\left[\begin{smallmatrix} a&0\\ 0&d\end{smallmatrix}\right]$, with $a\leq d$, $a+d>0$ (in particular, $[A,B]=aB$, $[A,C]=dC$),
\item[ ]
\item[(ii)] or $A_1=\left[\begin{smallmatrix} a&b\\ -b&a\end{smallmatrix}\right]$, with $a>0$, $b\ne 0$ (in particular, $[A,B]=aB-bC$,  $[A,C]=bB+aC$).
\end{itemize}

\begin{example}\label{bryant}
We now present in the format $\mu_{B}:=(A_1,A,B,C)$ the example given by R. Bryant in \cite[Example 1]{Bry}, as well as in \cite[Section 6.3]{ClyIvn} and \cite[Examples 4.13, 4.10]{LS-ERP}.   Consider,
$$
A_1=\left[\begin{smallmatrix} \frac{1}{3}&\\ &\frac{1}{3}\end{smallmatrix}\right], \quad
A=\left[\begin{smallmatrix}
-\frac{1}{6}&&& \\ &-\frac{1}{6}&& \\  &&\frac{1}{6}&\\ &&&\frac{1}{6}
\end{smallmatrix}\right],\quad
B=\left[\begin{smallmatrix}
&&0&0 \\ &&0&0 \\ 0&\frac{1}{3}&&\\ \frac{1}{3}&0&&
\end{smallmatrix}\right], \quad
C=\left[\begin{smallmatrix}
&&0&0 \\ &&0&0 \\ \frac{1}{3}&0&&\\ 0&-\frac{1}{3}&&
\end{smallmatrix}\right].
$$
Note that $\ngo$ is $2$-step nilpotent.
\end{example}

The following is a new example with a $4$-step nilpotent nilradical of dimension $6$.

\begin{example}\label{M1}
Consider $\mu_{M1}:=(A_1,A,B,C)$, where
$$
A_1:=\frac{1}{30}\left[\begin{smallmatrix}
 \sqrt{30} & 0 \\
  0 & 2\sqrt{30}
\end{smallmatrix}\right], \qquad
A:=\frac{1}{60}\left[\begin{smallmatrix}
 -10- \sqrt{30} & 0 & -2\sqrt{5} & 0 \\
 0 & -10 + \sqrt{30} & 0 & -2\sqrt{5} \\
 -2\sqrt{5} & 0 & 10- \sqrt{30} & 0 \\
 0 & -2\sqrt{5} & 0 & 10+ \sqrt{30} \\
\end{smallmatrix}\right],
$$
$$
B:=\frac{1}{30}\left[\begin{smallmatrix}
 0 & - \sqrt{5} & 0 & 5- \sqrt{30} \\
 5\sqrt{5} & 0 & 5 & 0 \\
 0 & 5+\sqrt{30} & 0 & \sqrt{5} \\
 5  & 0 & -5 \sqrt{5} & 0 \\
\end{smallmatrix}\right], \qquad
C:=\frac{1}{30}\left[\begin{smallmatrix}
 -\sqrt{5} & 0 & 5- \sqrt{30} & 0 \\
 0 &  \sqrt{5} & 0 & -5+ \sqrt{30} \\
5+\sqrt{30} & 0 & \sqrt{5} & 0 \\
 0 & -5-\sqrt{30} & 0 & -\sqrt{5} \\
\end{smallmatrix}\right].
$$
\end{example}

\begin{remark}\label{notequiv}
It is worth pointing out that the five examples we have given in this section (i.e.\ Examples \ref{J}, \ref{M2}, \ref{M3}, \ref{bryant}, \ref{M1}) are pairwise non-equivalent (even up to scaling).  Indeed, the underlying solvable Lie groups are pairwise non-isomorphic, and since they are all completely solvable, the corresponding left-invariant metrics can never be isometric up to scaling (see \cite{Alk}).
\end{remark}

\section{Deformations and rigidity}\label{deform-sec}

We study in this section deformations and two notions of rigidity for ERP $G_2$-structures on Lie groups.

As in Section \ref{preli}, we fix a $7$-dimensional real vector space $\ggo$ endowed with a basis $\{ e_1,\dots,e_7\}$ and the positive $3$-form defined in \eqref{phiA}, whose associated inner product $\ip$ is the one making the basis $\{ e_i\}$ oriented and orthonormal.

Let $\lca\subset\Lambda^2\ggo^*\otimes\ggo$ denote the algebraic subset of all Lie brackets on $\ggo$ and for every $\mu\in\lca$, denote by $G_\mu$ the simply connected Lie group with Lie algebra $(\ggo,\mu)$.  Each $\mu\in\lca$ will be identified with the left-invariant $G_2$-structure determined by $\vp$ on $G_\mu$:
$$
\mu \longleftrightarrow (G_\mu,\vp).
$$
The isomorphism class of $\mu$, $\Gl_7(\RR)\cdot\mu$, therefore stands for the set of all left-invariant $G_2$-structures on $G_\mu$,  due to the equivariant equivalence,
$$
(G_{h\cdot\mu},\vp) \simeq (G_\mu,\vp(h\cdot,h\cdot,h\cdot)), \qquad\forall h\in\Gl_7(\RR).
$$
Thus one has in $\lca$, all together, all the Lie groups endowed with left-invariant $G_2$-structures.  Note that two elements in $\lca$ are equivariantly equivalent as $G_2$-structures if and only if they belong to the same $G_2$-orbit, and that they are in the same $\Or(7)$-orbit if and only if they are equivariantly isometric as Riemannian metrics.  Both assertions hold without the word `equivariantly' for completely real solvable Lie brackets.

In this light, the following $G_2$-invariant algebraic subsets,
\begin{equation}\label{defL}
\lca_c := \left\{\mu\in\lca:d_\mu\vp=0\right\}, \qquad \lca_{erp} := \left\{\mu\in\lca_c:d_\mu\tau_\mu=\tfrac{1}{6}|\tau_\mu|^2\vp + \tfrac{1}{6}\ast(\tau_\mu\wedge\tau_\mu)\right\},
\end{equation}
parametrize the spaces of all closed (or calibrated) and all ERP $G_2$-structures on Lie groups, respectively.  Thus the quotient
$$
\lca_{erp}/G_2
$$
parametrizes the set of all ERP $G_2$-structures on Lie groups, up to equivariant equivalence.  Note that a given Lie group $G_\mu$ admits a closed (resp.\ ERP) $G_2$-structure if and only if the orbit $\Gl_7(\RR)\cdot\mu$ meets $\lca_c$ (resp.\ $\lca_{erp}$).

A $C^1$ curve $\mu:(-\epsilon,\epsilon)\longrightarrow\Lambda^2\ggo^*\otimes\ggo$ is said to be a {\it deformation} (of ERP $G_2$-structures) if $\mu(t)\in\lca_{erp}$ for all $t$.  Examples of deformations are given by $\mu(t)=h(t)\cdot\mu$, where $\mu\in\lca_{erp}$ and $h(t)\in G_2$, which are trivial in the sense that the family $\{\mu(t)\}$ is in such case pairwise equivariantly equivalent.  Given $\mu\in\lca_{erp}$, let ${\mathcal T}_\mu\lca_{erp}$ denote the set of all velocities $\mu'(0)$ such that $\mu(t)$ is a deformation with $\mu(0)=\mu$ (notice that ${\mathcal T}_\mu\lca_{erp}$ is not necessarily a vector space).  It follows that,
$$
\ggo_2\cdot\mu \subset {\mathcal T}_\mu\lca_{erp} \subset \overline{T}_\mu\lca_{erp},
$$
where $\ggo_2\cdot\mu$ coincides with the tangent space $T_\mu (G_2\cdot\mu)$ and $\overline{T}_\mu\lca_{erp}$ is the vector space determined by the linearization of both the Jacobi condition and the remaining equations defining $\lca_{erp}$ given in \eqref{defL}.

It is therefore natural to call a $\mu\in\lca_{erp}$ {\it equivariantly rigid} when
$$
\ggo_2\cdot\mu= \overline{T}_\mu\lca_{erp}.
$$
However, it is worth pointing out that according to Proposition \ref{equivA}, there might exist linear deformations of the form $\mu(t)=\mu+t\mu_{D}$, where $D$ is a suitable derivation of $\mu$.  Such deformations are also trivial as $\mu(t)$ is equivalent to $\mu$ for all $t$, though in general they are not equivariantly equivalent.  This shows that weaker notions of rigidity should also come into play.

In the case when $\hg=\spann\{ e_1,\dots,e_6\}$ is an ideal of $\mu\in\lca$, one has that $\mu=\lambda+\mu_A$ as in Section \ref{Gmu} and $\mu\leftrightarrow(G_\mu,\vp)$ is indeed the structure we have studied in Sections \ref{closed} and \ref{erp}.  It follows from Proposition \ref{closed-mu} that the $G_2$-orbit of any $\mu\in\lca_c$ meets the algebraic subset
$$
\lca_{c,\hg} := \left\{\mu\in\lca_c : \mu(\ggo,\hg)\subset\hg, \; \tr{\ad_\mu{e_i}}=0, \; i=1,\dots,6\right\},
$$
and that the equivariant equivalence between non-unimodular elements in $\lca_{c,\hg}$ is determined by the group $U_{\hg}$ given in \eqref{Uh}.  In the same vein, Proposition \ref{ERP-mu} asserts that any ERP $G_2$-structure $\mu\in\lca_{erp}$ is equivariantly equivalent to an element in
$$
\lca_{erp,\hg,\tau} := \left\{\mu\in\lca_{c,\hg} : \tau_\mu=\tau\right\},
$$
where $\tau:=e^{12}-e^{56}$.  In this case, the subgroups $U_{\hg,\tau}, U_{\ggo_1,\tau}\subset G_2$ computed in Section \ref{subg-sec} are the ones providing equivariant equivalence among $\lca_{erp,\hg,\tau}$ in the non-unimodular and unimodular cases, respectively (see Proposition \ref{equiv}).

This motivates the study of deformations within $\lca_{erp,\hg,\tau}$.  Analogously, for each $\mu\in\lca_{erp,\hg,\tau}$ one has that,
$$
\ug\cdot\mu = T_\mu (U\cdot\mu) \subset {\mathcal T}_\mu\lca_{erp,\hg,\tau} \subset \overline{T}_\mu\lca_{erp,\hg,\tau},
$$
where $\ug$, $U$ are either $\ug_{\hg,\tau}$, $U_{\hg,\tau}$ or $\ug_{\ggo_1,\tau}$, $U_{\ggo_1,\tau}$, depending on whether $\mu$ is non-unimodular or unimodular.   Here $\overline{T}_\mu\lca_{erp,\hg,\tau}$ is the linearization of the conditions defining $\lca_{erp,\hg,\tau}$ given in Proposition \ref{erp-iff}.  Thus $\mu\in\lca_{erp,\hg,\tau}$ is equivariantly rigid if and only if $\ug\cdot\mu=\overline{T}_\mu\lca_{erp,\hg,\tau}$.  Note that $\dim{\ug_{\hg,\tau}\cdot\mu}\leq 2$ and $\dim{\ug_{\ggo_1,\tau}\cdot\mu}\leq 4$ for any $\mu$.

According to the structural results proved in Section \ref{erp} (see Theorem \ref{erp-prop2} and Proposition \ref{erp-iff}), each $\mu\in\lca_{erp,\hg,\tau}$ only depends on one $2\times 2$ matrix $A_1$ and three $4\times 4$ matrices $A$, $B$ and $C$; in this way,
$$
\mu=\lambda_{B,C}+\mu_{\left[\begin{smallmatrix} A_1&0\\ 0&A\end{smallmatrix}\right]}.
$$
Thus any deformation $\mu(t)\in\lca_{erp,\hg,\tau}$ such that $\mu(0)=\mu$ and $\mu'(0)=\overline{\mu}$ has the following form:
$$
\mu=(A_1,A,B,C),  \quad \mu(t)=(A_1(t),A(t),B(t),C(t)),  \quad \overline{\mu} = (\overline{A}_1,\overline{A},\overline{B},\overline{C}),
$$
$$
A_1=\left[\begin{smallmatrix} a&b\\ c&d\end{smallmatrix}\right], \quad A_1(t)=\left[\begin{smallmatrix} a(t)&b(t)\\ c(t)&d(t)\end{smallmatrix}\right], \quad\overline{A}_1=\left[\begin{smallmatrix} \overline{a}&\overline{b}\\ \overline{c}&\overline{d}\end{smallmatrix}\right].
$$
It follows from Proposition \ref{erp-iff} that a vector $\overline{\mu}=(\overline{A}_1,\overline{A},\overline{B},\overline{C})$ belongs to $\overline{T}_\mu\lca_{erp,\hg,\tau}$ if and only if the following conditions hold:
\begin{align}
&\overline{A},\overline{B},\overline{C} \in \spg(\ggo_1,\tau), \label{tb1}\\
&[\overline{A},B] + [A,\overline{B}] = \overline{a}B + a\overline{B} + \overline{c}C + c\overline{C}, \label{tb2}\\
&[\overline{A},C] + [A,\overline{C}] = \overline{b}B + b\overline{B} + \overline{d}C + d\overline{C}, \label{tb3}\\
&[\overline{B},C] + [B,\overline{C}] = 0, \label{tb4}\\
&\theta(\overline{A}^t)\omega_7+\theta(\overline{B}^t)\omega_3+\theta(\overline{C}^t)\omega_4 = -(\overline{a}+\overline{d})\omega_7. \label{tb5}
\end{align}

We now describe the linear deformations mentioned above.  Given $\mu=(A_1,A,B,C)\in\lca_{erp,\hg,\tau}$, consider the vector space ${\mathfrak D}_\mu$ of all pairs $(D_1,D_2)\in\glg_2(\RR)\times\glg_4(\RR)$ such that
$$
[D_1,A_1]=0, \quad [D_2,A]=0, \quad [D_2,B]=rB+tC, \quad [D_2,C]=sB+uC, \quad D_1=\left[\begin{smallmatrix} r&s\\ t&u\end{smallmatrix}\right],
$$
that is,
$$
D:=\left[\begin{smallmatrix} D_1&0\\ 0&D_2\end{smallmatrix}\right]
$$
defines a derivation of $\mu$ vanishing at $e_7$ (see Remark \ref{equivA-rem}).  It follows from Proposition \ref{equivA} that each of these $(D_1,D_2)$ satisfying that $D\in\sug(3)$ determines a linear deformation of $\mu$ given by $\mu(t):=\mu+t\mu_D$, or equivalently,
$$
A_1(t)=A_1+tD_1, \qquad A(t)=A+tD_2, \qquad B(t)\equiv B, \qquad C(t)\equiv C,
$$
forming the vector space
$$
\dg_\mu:=\left\{\overline{\mu}=(D_1,D_2,0,0) : (D_1,D_2)\in{\mathfrak D}_\mu, \; D\in\sug(3) \right\} \subset {\mathcal T}_\mu\lca_{erp,\hg,\tau}.
$$
This suggests the following weaker version of rigidity: $\mu\in\lca_{erp,\hg,\tau}$ is said to be {\it rigid} if
$$
\dg_\mu + \ug\cdot\mu = \overline{T}_\mu\lca_{erp,\hg,\tau}.
$$

In the unimodular case, one always has that $\dg_\mu=0$ and $\ug_{\ggo_1,\tau}\cdot\mu$ is $4$-dimensional.  Indeed, any skew-symmetric derivation $D$ of $\mu=(0,A,B,C)$ must stabilize the nilradical $\ggo_1$ (see Proposition \ref{erp-unim}) and commute with the maximal abelian subalgebra $\spann\{ A,B,C\}\subset\sym(4)$ (see Corollary \ref{erp-unim-2}), so $D=0$.  This implies that a unimodular $\mu\in\lca_{erp,\hg,\tau}$ is equivariantly rigid, if and only if it is rigid, if and only if $\dim{\overline{T}_\mu\lca_{erp,\hg,\tau}}=4$.

In the non-unimodular case, ${\mathfrak D}_\mu=\Der(\mu)\cap\ggo_2$ and it is easy to see that $\dg_\mu\perp\ug_{\hg,\tau}\cdot\mu$.  Moreover, since $\Der(\mu)\cap\ug_{\hg,\tau}\subset\dg_\mu$, one always has that $\dim{(\dg_\mu + \ug_{\hg,\tau}\cdot\mu)}\geq 2$.

By solving the linear system \eqref{tb1}-\eqref{tb5} and computing the derivations belonging to $\ggo_2$ for all the examples given in Section \ref{cla-sec}, we obtain the following information:
\begin{itemize}
\item $\mu_J$ (Example \ref{J}): $\ug_{\ggo_1,\tau}\cdot\mu = \overline{T}_\mu\lca_{erp,\hg,\tau}$ ($4$-dimensional), $\dg_\mu=0$.

\item $\mu_{M2}$ (Example \ref{M2}): $\ug_{\hg,\tau}\cdot\mu = \overline{T}_\mu\lca_{erp,\hg,\tau}$ ($2$-dimensional), $\dg_\mu=0$.

\item $\mu_{M3}$ (Example \ref{M3}): $\ug_{\hg,\tau}\cdot\mu = \overline{T}_\mu\lca_{erp,\hg,\tau}$ ($2$-dimensional), $\dg_\mu=0$.

\item $\mu_B$ (Example \ref{bryant}): $\ug_{\hg,\tau}\cdot\mu = 0$, $\dg_\mu=\overline{T}_\mu\lca_{erp,\hg,\tau}$ ($2$-dimensional).

\item $\mu_{M1}$ (Example \ref{M1}): $\ug_{\hg,\tau}\cdot\mu = \overline{T}_\mu\lca_{erp,\hg,\tau}$ ($2$-dimensional), $\dg_\mu=0$.
\end{itemize}

It follows that they are all equivariantly rigid, except for Example \ref{bryant}, which is only rigid.

\end{document}